\documentclass[11pt]{article}

\usepackage[a4paper,margin=29mm]{geometry}
\usepackage{amsmath,amssymb,amsthm,mathtools}
\usepackage{booktabs}
\usepackage{microtype}
\usepackage{needspace}
\usepackage[hidelinks]{hyperref}

\newtheorem{theorem}{Theorem}[section]
\newtheorem{proposition}[theorem]{Proposition}
\newtheorem{lemma}[theorem]{Lemma}
\newtheorem{corollary}[theorem]{Corollary}
\newtheorem{remark}[theorem]{Remark}
\theoremstyle{definition}

\newcommand{\Lam}{\Lambda}
\newcommand{\Lstar}{\Lambda_*}

\newcommand{\dd}{\,\mathrm d}

\title{Finite Hardy Phase Theory for Power Means\\
\large An exact analytic shooting family through the Carleman parameter}
\author{Tianchi Huang\\[2pt]\small Independent researcher}
\date{28 September 2026}

\hypersetup{
  pdftitle={Finite Hardy Phase Theory for Power Means: An exact analytic shooting family through the Carleman parameter},
  pdfauthor={Tianchi Huang},
  pdfsubject={Finite Hardy inequalities and asymptotic phase theory},
  pdfkeywords={Hardy inequality, power means, finite sections, Carleman inequality, asymptotic phase, Poincar\'e expansion}
}

\begin{document}

\maketitle

\begin{abstract}
For the power mean $P_t$, $t<1$, let
\[
 \Lam_N(t)=\sup_{x_k>0}
 \frac{\sum_{n=1}^N P_t(x_1,\ldots,x_n)}{\sum_{n=1}^N x_n}
\]
be its finite Hardy constant.  Negative powers, the geometric mean, and the
usual positive-exponent Hardy inequality correspond respectively to $t<0$,
$t=0$, and $0<t<1$.  With the analytic parameter
$q=t/(1-t)\in(-1,\infty)$, and writing
$\lambda_N(q)=\Lam_N(q/(1+q))$, we derive one exact scalar shooting map for all
three regimes; its apparent singularity at $q=0$ is removable and gives the
finite Carleman problem exactly.  The map has a common terminal condition and
its limiting vector field has a quadratic critical bottleneck at
\[
 y_*(q)=1+q,\qquad
 \Lstar(q)=(1+q)^{(1+q)/q}.
\]
We prove that there is a unique real-analytic phase $\kappa(q)$ for which the
finite-defect expansion has zero cubic coefficient and, uniformly for $q$ in
compact subsets of $(-1,\infty)$, for every fixed $L\ge2$,
\[
 \Lstar(q)-\lambda_N(q)
 =\sum_{j=2}^{L}\frac{A_j(q)}{(\log N+\kappa(q))^j}
 +O\!\left((\log N)^{-L-1}\right).
\]
In this phase-normalized scale,
\[
 A_2=2\pi^2(1+q)\Lstar,\qquad A_3=0,
 \qquad
 A_4=-\frac{\pi^2}{3}(2q^2+5q+5)A_2.
\]
Thus the classical Carleman and positive-Hardy leading corrections are
projections of a single analytic family.  At the Carleman parameter we obtain
$A_2(0)=2e\pi^2$ and $A_4(0)=-(10/3)e\pi^4$; the latter is also recovered by a
direct local residue calculation.  The phase itself contains the global
discrete defect.  After phase normalization, the first coefficient involving
a nonconstant discrete response is $A_5$, at order
$(\log N+\kappa(q))^{-5}$.
\end{abstract}



\section{Introduction and main results}

For $t<1$, define
\begin{equation}\label{eq:power-mean}
 P_t(x_1,\ldots,x_n)=
 \begin{cases}
 \displaystyle\left(\frac1n\sum_{k=1}^n x_k^t\right)^{1/t},&t\ne0,\\[1.2ex]
 \displaystyle(x_1\cdots x_n)^{1/n},&t=0.
 \end{cases}
\end{equation}
Following the finite Hardy-sequence terminology of P\'ales and Pasteczka
\cite{PalesPasteczka2016}, set
\begin{equation}\label{eq:finite-constant}
 \Lam_N(t)=H_N(P_t)
 =\sup_{x_k>0}
 \frac{\sum_{n=1}^N P_t(x_1,\ldots,x_n)}{\sum_{n=1}^N x_n}.
\end{equation}
The infinite Hardy constant is classical:
\begin{equation}\label{eq:critical-constant-t}
 \Lam_\infty(t)=(1-t)^{-1/t},
\end{equation}
with the continuous value $e$ at $t=0$.
For $t=1/p\in(0,1)$ this is Hardy's classical constant after the
substitution $x_k=a_k^p$; see, for example,
\cite[Theorem~326]{HardyLittlewoodPolya1952}.  The limit $t=0$ is the
Carleman case.

The finite object in \eqref{eq:finite-constant} is classical.  De Bruijn
obtained the sharp leading correction at the geometric-mean point
\cite{deBruijn1963}.  Gao's historical discussion attributes an earlier
positive-branch asymptotic to Ackermans \cite{Ackermans1964,GaoHardy}; the
quantitative comparison in Section~\ref{sec:projections} uses Gao's own
theorem, not that secondary transcription.  Wilf's monograph \cite{Wilf1970}
gives a standard account of finite sections of classical inequalities, and
later weighted variants were developed in \cite{GaoCarleman,GaoHardy}.  On the
negative-power side, finite weighted discrete refinements appear in
\cite{Cizmesija2005}, whereas \cite{Nikolidakis2014} concerns related integral
Hardy inequalities and Muckenhoupt weights.  The matched-weight Hardy
functional for general weighted means was developed by P\'ales and Pasteczka
\cite{PalesPasteczkaWeighted}; it is distinct from the weighted-mean-matrix
finite sections considered by Gao.  Exact nonlinear finite systems for the
Carleman constant also have precedent \cite{WuZhangWang2008}.  Finite
Hilbert-space sections have been studied by spectral methods
\cite{DimitrovGadjevNikolovUluchev2021,GadjevGochev2023,Stampach2022}; at
$p=2$, detailed higher-order expansions appear in
\cite{DimitrovGadjevIsmail2024,Stampach2022}.  Their published fifth-order
coefficients are not consistent with one another; Section~\ref{sec:projections}
resolves the discrepancy by retaining the cubic term in the exact Gamma
phase.  Our purpose is therefore not to reintroduce
finite Hardy constants, shooting, or single-parameter asymptotics.  The
structural contribution is that the full power-mean range forms one exact
analytic discrete family, including the Carleman parameter, with a
parameter-uniform phase construction and all-order Poincar\'e asymptotics.

The natural coordinate is
\begin{equation}\label{eq:q-parameter}
 q=\frac{t}{1-t}\in(-1,\infty),\qquad
 t=\frac{q}{1+q}.
\end{equation}
To keep the two parameterizations distinct, write
\begin{equation}\label{eq:lambda-q}
 \lambda_N(q):=\Lam_N\!\left(\frac{q}{1+q}\right)
\end{equation}
and
\begin{equation}\label{eq:critical-q}
 y_*(q)=1+q,\qquad
 \Lstar(q)=(1+q)^{(1+q)/q},
\end{equation}
where $\Lstar(0)=e$.  Thus
$\Lstar(q)=\Lam_\infty(q/(1+q))$; throughout the paper $\Lstar$ denotes
the critical constant in the $q$ parameter, while $\Lam_\infty$ uses $t$.

Our first result identifies the finite variational problem with a unique
admissible scalar shot.

\Needspace{0.60\textheight}
\begin{theorem}[Finite extremizer and exact shooting]\label{thm:shooting}
For every $N\ge1$ and $q\in(-1,\infty)$, the maximum in
\eqref{eq:finite-constant}, after normalization $\sum x_k=1$, is attained at a
unique point of the open simplex.  Removing the normalization gives one
positive extremal ray.

Let $h_n=1/n$ and, for $q\ne0$, define
\begin{equation}\label{eq:exact-map}
 T_{n,q,\Lam}(y)
 =(y-h_n)
 \left[
 \frac{1+h_n}
 {1+h_n((y-h_n)/\Lam)^q}
 \right]^{1/q}.
\end{equation}
At $q=0$ this has the removable continuation
\begin{equation}\label{eq:exact-map-zero}
 T_{n,0,\Lam}(y)
 =(y-1/n)^{n/(n+1)}\Lam^{1/(n+1)}.
\end{equation}
For real $q>-1$ and $\Lam>0$, set
\[
 \operatorname{Dom}(T_{n,q,\Lam})=(1/n,\infty).
\]
An orbit is \emph{admissible through $N$} if
$y_n\in\operatorname{Dom}(T_{n,q,\Lam})$ for $1\le n<N$; the endpoint
$y_N=1/N$ is allowed as the terminal boundary value.
Then $\lambda_N(q)$ is the unique value $\Lam$ for which the orbit
\begin{equation}\label{eq:shooting-orbit}
 y_1=\Lam,\qquad y_{n+1}=T_{n,q,\Lam}(y_n)
\end{equation}
is admissible through $N$ and satisfies
\begin{equation}\label{eq:terminal}
 Ny_N=1.
\end{equation}
The map $(q,\Lam,y,h)\mapsto T_{h,q,\Lam}(y)$ is locally jointly real analytic
across $q=0$ whenever $\Lam>0$, $h>0$, and $y>h$ (and its complexification is
holomorphic on the corresponding local charts).
\end{theorem}

The principal theorem is the following compact-parameter phase expansion.

\begin{theorem}[Uniform finite Hardy phase theorem]\label{thm:main}
There exist real-analytic functions $\kappa$ and $A_j$, $j\ge2$, on
$(-1,\infty)$ such that, for every compact set $Q\Subset(-1,\infty)$ and
every fixed $L\ge2$,
\begin{equation}\label{eq:main-expansion}
 \Lstar(q)-\lambda_N(q)
 =\sum_{j=2}^{L}
 \frac{A_j(q)}{(\log N+\kappa(q))^j}
 +O_{Q,L}((\log N)^{-L-1})
\end{equation}
uniformly for $q\in Q$.  The normalization $A_3\equiv0$ uniquely determines
the phase $\kappa$ and then all the coefficient functions $A_j$ among
families satisfying these expansions on the whole interval.  The first
coefficients are
\begin{align}
 A_2(q)&=\mathcal K(q)=2\pi^2(1+q)\Lstar(q),\label{eq:Kq}\\
 A_3(q)&=0,\label{eq:A3}\\
 A_4(q)&=-\frac{\pi^2}{3}(2q^2+5q+5)\mathcal K(q).
 \label{eq:A4q}
\end{align}
No uniformity is asserted as $q\downarrow-1$ or $q\to\infty$.
\end{theorem}

Here and throughout, ``complete expansion'' means complete in the fixed-order
Poincar\'e sense stated in \eqref{eq:main-expansion}: for each fixed $L$ the
remainder has the displayed order as $N\to\infty$.  It makes no assertion that
the infinite series obtained by letting $L\to\infty$ converges.

The three familiar regimes are now substitutions in one theorem.

\begin{corollary}[Three projections]\label{cor:projections}
In the $t$ parameter, with continuous values at $t=0$, the compositions of
the $q$-coefficient functions are
\begin{align}
 A_2\!\left(\frac{t}{1-t}\right)
 &=2\pi^2(1-t)^{-1/t-1},\label{eq:A2t}\\
 A_4\!\left(\frac{t}{1-t}\right)
 &=-\frac{2\pi^4}{3}
 (2t^2-5t+5)(1-t)^{-1/t-3}.\label{eq:A4t}
\end{align}
In particular:
\begin{enumerate}
 \item $t=1/p\in(0,1)$ gives the positive-exponent finite Hardy constants;
 \item $t=0$ gives
 \begin{equation}\label{eq:carleman-expansion}
 e-\Lam_N(0)
 =\frac{2e\pi^2}{(\log N+\kappa(0))^2}
 -\frac{10e\pi^4/3}{(\log N+\kappa(0))^4}
 +O((\log N)^{-5});
 \end{equation}
 the fourth-order coefficient here is also obtained directly, without taking
 a parameter limit, in Lemma~\ref{lem:carleman-residue};
 \item $t=-1/r<0$ gives the negative-exponent branch.
\end{enumerate}
At $p=2$ (hence $q=1$) one recovers $A_2(1)=16\pi^2$ and
$A_4(1)=-64\pi^4$, in agreement with the Hilbert-space expansion
\cite{Stampach2022}; the fifth-order comparison is discussed in
Section~\ref{sec:projections}.
\end{corollary}

The remainder of the paper proves Theorems \ref{thm:shooting} and
\ref{thm:main}.  The exact map is derived in Section \ref{sec:variational};
the common critical geometry is isolated in Section \ref{sec:geometry}; an
abstract phase theorem is proved in Section \ref{sec:abstract}; and its
parameter-uniform verification for the power-mean family is given in Section
\ref{sec:verification}.  Section \ref{sec:coefficients} computes the explicit
coefficients, including a calculation made directly at $q=0$.

\section{Finite variation and the exact analytic family}
\label{sec:variational}

Put $M_n=P_t(x_1,\ldots,x_n)$.  Homogeneity allows us to maximize
$\sum_{n=1}^N M_n$ over the closed simplex $\sum x_k=1$.  For $t<0$, we use
the continuous convention that a prefix mean containing a zero coordinate is
zero.  This is the unique continuous extension from the positive orthant: if
$m=\min_{1\le k\le n}x_k$, then
\begin{equation}\label{eq:negative-boundary-control}
 0<P_t(x_1,\ldots,x_n)\le n^{-1/t}m,
 \qquad t<0.
\end{equation}
Indeed, $n^{-1}\sum x_k^t\ge m^t/n$, and raising to the negative power $1/t$
reverses the inequality.  Thus the mean tends uniformly to zero whenever any
coordinate tends to zero, including at the origin and along paths with
different rates.

\begin{proposition}[Interior extremizer and calibration]\label{prop:extremizer}
The maximizer is interior and satisfies
\begin{equation}\label{eq:EL}
 \Lam x_k^{1-t}
 =\sum_{n=k}^N\frac{M_n^{1-t}}{n},
 \qquad 1\le k\le N.
\end{equation}
Conversely, every positive vector satisfying \eqref{eq:EL} calibrates the
global maximum.  Its positive ray is unique.
\end{proposition}

\begin{proof}
Existence now follows from continuity on the compact simplex.  For $0<t<1$,
adding $\varepsilon$ at the first zero coordinate produces a gain of order
$\varepsilon^t$, which dominates the linear normalization loss.  At $t=0$, a
boundary point with $x_1=0$ has every prefix geometric mean equal to zero and
therefore cannot maximize.  If the first zero coordinate is $j\ge2$, the first
newly nonzero geometric mean gives a gain of order $\varepsilon^{1/j}$, which
again dominates the linear normalization loss.  For $t<0$, an induction on
the length uses
\[
 P_t(x_1,\ldots,x_k)\le k^{-1/t}x_k
\]
and the explicit strict estimate
\begin{equation}\label{eq:finite-negative-upper}
 \Lam_{k-1}(t)
 \le \max_{1\le n\le k-1}n^{-1/t}
 =(k-1)^{-1/t}<k^{-1/t}.
\end{equation}
Indeed, the displayed pointwise bound holds with any chosen coordinate of a
prefix.  Applying it to the last entry of each prefix gives
$\sum_{n<k}P_t(x_1,\ldots,x_n)\le(k-1)^{-1/t}\sum_{n<k}x_n$.
If a normalized length-$(k-1)$ extremizer is extended by $\varepsilon>0$, then
\[
 P_t(x_1,\ldots,x_{k-1},\varepsilon)
 =k^{-1/t}\varepsilon+o(\varepsilon),
\]
so its quotient increases by
$(k^{-1/t}-\Lam_{k-1}(t))\varepsilon+o(\varepsilon)$.  For a boundary
candidate whose first zero is $j$, all later prefix means are zero.  If
$j=1$, its objective is zero and it cannot maximize; if $j\ge2$, its value is
at most $\Lam_{j-1}(t)<\Lam_N(t)$, again a contradiction.

At an interior maximizer, differentiating under the simplex constraint and
using
\[
 \partial_{x_k}M_n=\frac1nM_n^{1-t}x_k^{t-1}
\]
gives a common Lagrange multiplier.  Euler's identity for each homogeneous
mean is
\[
 \sum_kx_k\partial_{x_k}M_n=M_n.
\]
Multiplying the stationarity equations by $x_k$ and summing first in $k$ and
then in $n$ shows that the multiplier is the extremal value $\Lam$; hence
\eqref{eq:EL} follows.
Conversely, the tangent-plane inequality for the concave
means $P_t$, summed in $n$, proves global optimality.  For completeness, when
$t\ne0,1$ their Hessian satisfies
\[
 D^2P_t(x)[h,h]
 =-\frac{1-t}{n^2}P_t(x)^{1-2t}
 \left[
  \left(\sum x_i^t\right)\left(\sum h_i^2x_i^{t-2}\right)
  -\left(\sum h_ix_i^{t-1}\right)^2
 \right].
\]
The bracket is nonnegative by Cauchy--Schwarz, and equality holds precisely
when $h$ is radial.  The limiting statement at $t=0$ is the same.  Hence the
equality case in concavity forces any two positive maximizers to be
proportional; the simplex normalization removes that freedom.
The same boundary perturbations show that
$\Lam_{N+1}(t)>\Lam_N(t)$: for $0<t<1$ one may first append a zero (the new
prefix mean is already positive), while for $t=0$ and $t<0$ the perturbations
used above apply at the appended coordinate.
\end{proof}

\begin{proof}[Proof of Theorem \ref{thm:shooting}]
Define
\begin{equation}\label{eq:U-y}
 U_n=\frac{x_n}{M_n},\qquad
 y_n=\Lam U_n^{1-t},\qquad q=\frac{t}{1-t}.
\end{equation}
Subtracting the equations \eqref{eq:EL} at $k=n$ and $k=n+1$ gives
\begin{equation}\label{eq:EL-difference}
 \Lam\bigl(x_n^{1-t}-x_{n+1}^{1-t}\bigr)
 =\frac{M_n^{1-t}}{n}.
\end{equation}
Divide \eqref{eq:EL-difference} by $M_n^{1-t}$ and put
\[
 R_n=\frac{x_{n+1}}{M_n}.
\]
Then
\[
 R_n^{1-t}=\frac{y_n-h_n}{\Lam},
 \qquad
 R_n^t=\left(\frac{y_n-h_n}{\Lam}\right)^q.
\]
The exact power-mean update gives
\[
 \frac{M_{n+1}}{M_n}
 =\left(\frac{n+R_n^t}{n+1}\right)^{1/t},
\]
and consequently
\begin{align*}
 y_{n+1}
 &=\Lam\left(\frac{x_{n+1}}{M_{n+1}}\right)^{1-t}\\
 &=(y_n-h_n)
 \left(\frac{n+1}{n+R_n^t}\right)^{(1-t)/t}.
\end{align*}
Since $(1-t)/t=1/q$, this is \eqref{eq:exact-map}.  The first
equation in \eqref{eq:EL} gives $y_1=\Lam$.  The last one gives
$Ny_N=1$, while the unremoved positive tail in \eqref{eq:EL} gives
$ny_n>1$ for $n<N$.

For the continuation at $q=0$, write the bracket in \eqref{eq:exact-map} as
an exponential and use
\[
 \lim_{q\to0}\frac{
 \log(1+h)-\log(1+h\exp(qL))}{q}
 =-\frac{h}{1+h}L,
 \qquad L=\log\!\left(\frac{y-h}{\Lam}\right).
\]
Since $h/(1+h)=1/(n+1)$, this is \eqref{eq:exact-map-zero}.  The same
exponential form proves joint analyticity after removing the factor $q$.
Conversely, suppose an admissible shot reaches $Ny_N=1$.  Put
\[
 U_n=(y_n/\Lam)^{1/(1-t)},\qquad
 R_n=((y_n-h_n)/\Lam)^{1/(1-t)}.
\]
Choose $M_1=x_1>0$ arbitrarily and recursively set
$x_{n+1}=R_nM_n$ and
\[
 M_{n+1}=\left(\frac{nM_n^t+x_{n+1}^t}{n+1}\right)^{1/t},
\]
with the geometric interpretation at $t=0$.  The exact map says precisely
that $x_{n+1}/M_{n+1}=U_{n+1}$, so the reconstructed $M_n$ are the prefix
means and \eqref{eq:EL-difference} holds at every step.  The terminal identity
is the $k=N$ equation in \eqref{eq:EL}; summing the difference equations
backwards therefore recovers all of \eqref{eq:EL}.  Proposition
\ref{prop:extremizer} identifies the shot with the unique extremal ray, and
the admissible shooting parameter is unique.
\end{proof}

The terminal variable is already unified:
\begin{equation}\label{eq:X-variable}
 X_n=ny_n,\qquad X_n>1\ (n<N),\qquad X_N=1.
\end{equation}
The apparently opposite terminal behaviour in the standard positive- and
negative-exponent variables is therefore a coordinate artefact.

\section{The common critical bottleneck}
\label{sec:geometry}

Introduce the $q$-logarithm
\begin{equation}\label{eq:qlog}
 \log_qz=\begin{cases}(z^q-1)/q,&q\ne0,\\ \log z,&q=0,\end{cases}
\end{equation}
and the vector field
\begin{equation}\label{eq:B-field}
 B_{q,\Lam}(y)=1+y\log_q(y/\Lam).
\end{equation}
On every range
$q\in Q\Subset(-1,\infty)$,
$0<\Lam_-\le\Lam\le\Lam_+<\infty$, and $0<y\le Y<\infty$, a uniform
expansion of the exact map gives
\begin{equation}\label{eq:map-field}
 y_{n+1}=y_n-\frac1nB_{q,\Lam}(y_n)
 +O_Q\!\left(\frac{1}{n^2y_n}\right)
\end{equation}
whenever $1/n\le\sigma y_n$ for a sufficiently small fixed $\sigma>0$.
The implied constant depends only on the displayed compact range and on
$\sigma$; below those auxiliary bounds are chosen from $Q$ and are suppressed
from the notation.
Thus the logarithmic-time equation is $\dot y=-B_{q,\Lam}(y)$.

\begin{lemma}[Analytic double root]\label{lem:double-root}
For every $q>-1$, the data \eqref{eq:critical-q} satisfy
\begin{equation}\label{eq:double-root}
 B_{q,\Lstar}(y_*)=0,\qquad
 \partial_yB_{q,\Lstar}(y_*)=0,\qquad
 \partial_y^2B_{q,\Lstar}(y_*)=\frac1{1+q}>0.
\end{equation}
These identities, and all coefficients of the translated vector field, are
analytic across $q=0$.
\end{lemma}

\begin{proof}
For $q\ne0$ the identities follow by direct differentiation of
\eqref{eq:B-field}; their limits at $q=0$ follow from
$B_{0,e}(y)=1+y\log(y/e)$.  Analyticity follows from the integral identity
\[
 \log_qz=\int_0^1z^{\theta q}\log z\,\dd\theta.
\]
\end{proof}

Let
\begin{equation}\label{eq:normal-variables}
 \delta=\Lstar(q)-\Lam,\qquad v=y_*(q)-y.
\end{equation}
Then the oriented field
\begin{equation}\label{eq:g-field}
 g_{q,\delta}(v)=B_{q,\Lstar(q)-\delta}(y_*(q)-v)
\end{equation}
has the uniform normal form
\begin{equation}\label{eq:normal-form}
 g_{q,\delta}(v)
 =a(q)\delta+b(q)v^2
 +O_Q(\delta^2+\delta|v|+|v|^3),
\end{equation}
where
\begin{equation}\label{eq:ab}
 a(q)=\Lstar(q)^{-1},\qquad b(q)=\frac1{2(1+q)}.
\end{equation}
Both sides of the Carleman point therefore cross the same oriented
saddle-node bottleneck; no sign reversal is required in the $y$ coordinate.

\begin{lemma}[Persistence of the scaled-step condition]\label{lem:scaled-step}
Let $Q\Subset(-1,\infty)$.  There are $r,\delta_*,\sigma>0$, an integer
$n_*$, and $c,C>0$, depending only on $Q$, with the following property.
For an admissible step of the translated exact map, put
$\rho_n=\delta+v_n^2$ and $\Delta_n=v_{n+1}-v_n$.
If $q\in Q$, $0\le\delta\le\delta_*$, $n\ge n_*$, $|v_n|\le r$, and
$n^{-1}\le\sigma\rho_n$, then
\[
 0<\frac{c\rho_n}{n}\le\Delta_n\le\frac{C\rho_n}{n},
 \qquad (n+1)\rho_{n+1}\ge n\rho_n.
\]
Consequently the scaled-step condition, once valid at an entry index,
persists at every subsequent index while the orbit remains in $|v|\le r$.
\end{lemma}

\begin{proof}
On a preliminary fixed small neighbourhood, \eqref{eq:map-field} and
\eqref{eq:normal-form} give
\[
 \Delta_n=\frac{g_{q,\delta}(v_n)}n+O_Q(n^{-2}),
 \qquad c_0\rho_n\le g_{q,\delta}(v_n)\le C_0\rho_n.
\]
Choose $\sigma$ small enough to absorb the $O_Q(n^{-2})$ term whenever
$n^{-1}\le\sigma\rho_n$.  This gives the displayed drift bounds with
fixed $c,C$.  Shrink $r$ further so that $4Cr\le1$, without increasing $C$.
Since $\Delta_n>0$,
\[
 \rho_{n+1}=\rho_n+2v_n\Delta_n+\Delta_n^2
 \ge\rho_n\left(1-\frac{2Cr}{n}\right).
\]
For $n\ge1$,
\[
 (n+1)\left(1-\frac{2Cr}{n}\right)-n
 =1-2Cr\left(1+\frac1n\right)\ge0.
\]
Thus $(n+1)\rho_{n+1}\ge n\rho_n\ge\sigma^{-1}$, which is the step
condition at $n+1$.  Induction proves persistence, including within each
incoming shell.
\end{proof}

\begin{lemma}[Critical capture and coarse scale]\label{lem:capture}
Let $Q\Subset(-1,\infty)$.  The critical orbit of
\eqref{eq:shooting-orbit} with $\Lam=\Lstar(q)$ exists globally and satisfies
\begin{equation}\label{eq:critical-orbit}
 y_n^*(q)>y_*(q),\qquad y_{n+1}^*(q)<y_n^*(q),
\end{equation}
and, uniformly for $q\in Q$,
\begin{equation}\label{eq:riccati-asymptotic}
 y_n^*(q)-y_*(q)
 =\frac{2(1+q)}{\log n}
 +O_Q\!\left(\frac{\log\log n}{(\log n)^2}\right).
\end{equation}
Moreover there are $c_Q,C_Q>0$ such that, for every $N\ge2$,
\begin{equation}\label{eq:coarse-scale}
 \frac{c_Q}{(\log N)^2}
 \le \Lstar(q)-\lambda_N(q)
 \le \frac{C_Q}{(\log N)^2}.
\end{equation}
\end{lemma}

\begin{proof}
We divide the proof into four steps.

\emph{Step 1: global admissibility and exact monotonicity.}
Write $h=1/n$ and $z=(y-h)/\Lam$.  Direct differentiation of
\eqref{eq:exact-map}, with the continuous values at $q=0$, gives
\begin{equation}\label{eq:T-monotone}
 \partial_y\log T_{n,q,\Lam}(y)
 =\frac1{(y-h)(1+hz^q)}>0,
 \qquad
 \partial_\Lam\log T_{n,q,\Lam}(y)
 =\frac{hz^q}{\Lam(1+hz^q)}>0.
\end{equation}
The proof of Proposition \ref{prop:extremizer} also shows that the finite
constants are strictly increasing.  Their limit is the infinite Hardy
constant: applying the finite inequalities to truncations gives one direction,
while applying the infinite inequality to summable positive extensions of a
finite vector and then letting the tail mass tend to zero gives the other.
Consequently $\lambda_N(q)<\Lstar(q)$ and
$\lambda_N(q)\uparrow\Lstar(q)$.  This convergence is uniform on $Q$ by Dini's
theorem; continuity of each $\lambda_N$ follows from the compact variational
problem and the joint continuous boundary extension of the means.  If
$y_k^{(N)}$ denotes the extremal orbit of length $N$, then
\eqref{eq:T-monotone} gives, inductively,
\[
 y_k^*>y_k^{(N)}>1/k\quad (k<N),
 \qquad y_N^*>y_N^{(N)}=1/N.
\]
Thus the shot starting from $y_1^*=\Lstar$ never reaches the boundary of a
map domain.  Since $N$ is arbitrary, it defines a global admissible orbit.

The same calculation gives an exact one-sided direction.  Since
\[
 \partial_\Lam B_{q,\Lam}(y)
 =-\frac{y}{\Lam}\left(\frac y\Lam\right)^q<0,
\]
strict convexity and
\eqref{eq:double-root} imply
\begin{equation}\label{eq:B-positive-subcritical}
 B_{q,\Lam}(y)\ge B_{q,\Lstar}(y)\ge0,
 \qquad 0<\Lam\le\Lstar.
\end{equation}
Put $u=h/y\in(0,1)$.  For $q>0$, the inequality
$T_{n,q,\Lam}(y)<y$ is equivalent to
\[
 \frac{(1-u)^{-q}-1}{h}+(y/\Lam)^q>1.
\]
Indeed,
\[
 (1-u)^{-q}-1
 =q\int_0^u(1-s)^{-q-1}\,\dd s>qu=\frac{qh}{y}.
\]
On dividing by $h$, this makes the first term strictly larger than $q/y$,
whereas \eqref{eq:B-positive-subcritical} gives
$(y/\Lam)^q\ge1-q/y$.  For $-1<q<0$, the equivalent inequality is the same
display with ``$<$'' in place of ``$>$''; now
$(y/\Lam)^q\le1-q/y$.  Writing $r=-q\in(0,1)$ gives symmetrically
\[
 (1-u)^r-1
 =-r\int_0^u(1-s)^{r-1}\,\dd s<-ru=qu,
\]
because $(1-s)^{r-1}>1$ for $0<s<u$.  Thus
$((1-u)^{-q}-1)/h<q/y$.  At $q=0$,
\[
 (1+h)\log\frac{T_{n,0,\Lam}(y)}y
 =\log(1-u)+h\log(\Lam/y)\le\log(1-u)+u<0,
\]
again by \eqref{eq:B-positive-subcritical}.  Consequently every critical or
subcritical admissible orbit is strictly decreasing in the $y$ coordinate.

We now exclude a crossing of $y_*$.  If $y_j^*\le y_*$, strict decrease gives
$y_{j+1}^*<y_*$, and the subsequent orbit is decreasing below $y_*$.  Let
$\ell\ge0$ be its limit.  If $\ell>0$, the compact interval
$[\ell,y_{j+1}^*]$ excludes the only zero $y_*$ of the critical field, and
the expansion \eqref{eq:map-field} would give
$y_n^*-y_{n+1}^*$ bounded below by a fixed positive multiple of $1/n$, a
contradiction.  If $\ell=0$, choose
$A$ large.  As long as $X_n=ny_n^*>A$, the same expansion and the positive
minimum of $B_{q,\Lstar}$ on $[0,y_{j+1}^*]$ again give
$y_n^*-y_{n+1}^*$ bounded below by a fixed positive multiple of $1/n$; hence
the orbit must enter $1<X_n\le A$.
Lemma \ref{lem:terminal-layer} below then forces $X_n\le1$ after a bounded
number of further steps, contradicting global admissibility.  This use is not
circular: that lemma follows directly from the exact map and does not use the
present coarse-scale estimate.  We have proved $y_n^*>y_*$.  Its decreasing
limit cannot exceed $y_*$, since otherwise \eqref{eq:map-field} again supplies
a nonsummable harmonic decrement.  Therefore $y_n^*\downarrow y_*$.  Every
finite composition is continuous in $q$, the sequence is pointwise decreasing
in the index $n$, and the limit $y_*(q)$ is continuous.  Dini's theorem,
applied in $n$, therefore makes this convergence uniform for $q\in Q$.

\emph{Step 2: discrete Riccati barriers.}
Set $w_n=y_n^*-y_*>0$ and $\ell_n=1+\log n$.  Uniform Taylor expansion at the
double root gives
\begin{equation}\label{eq:critical-riccati}
 w_{n+1}=w_n-\frac{b(q)}n w_n^2+E_n,
 \qquad
 |E_n|\le C_Q\left(\frac{w_n^3}{n}+\frac1{n^2}\right),
 \qquad b(q)=\frac1{2(1+q)}.
\end{equation}
The additive $O(n^{-2})$ term prevents one from obtaining a uniform lower
barrier merely by shrinking a trial multiple of $\ell_n^{-1}$.  We first
construct an exact auxiliary barrier.  Put $s=1+q$ and
\[
 \underline y_n=s+\frac1n,
 \qquad
 f_q(x)=
 \begin{cases}
  (1-qx)^{-1/q},&q\ne0,\\
  e^x,&q=0.
 \end{cases}
\]
On the positive real domains used here,
\[
 f_q(0)=f_q'(0)=1,
 \qquad
 f_q''(x)=(1+q)(1-qx)^{-1/q-2}>0,
\]
with the continuous interpretation at $q=0$.  Hence $f_q(x)>1+x$ for
$x>0$.  The critical identities
$(s/\Lstar)^q=1/s$ and
$\Lstar=s f_q(1/s)$, also understood continuously at $q=0$, give directly
from the exact map
\begin{equation}\label{eq:exact-critical-barrier}
 T_{n,q,\Lstar}(\underline y_n)
 =s f_q\!\left(\frac1{s(n+1)}\right)
 >s+\frac1{n+1}=\underline y_{n+1},
 \qquad
 y_1^*=\Lstar>s+1=\underline y_1.
\end{equation}
Since the exact map is increasing, induction proves
$y_n^*>\underline y_n$ for every $n$.  Define
\[
 r_n(q)=y_n^*(q)-\underline y_n(q)>0.
\]
The uniform critical capture proved in Step~1 implies $r_n\to0$ uniformly on
$Q$.  Uniformly for $s+w$ between $\underline y_n$ and $y_n^*$, Taylor
expansion of the exact derivative gives
\[
 \partial_yT_{n,q,\Lstar}(s+w)
 =1-\frac1n\partial_yB_{q,\Lstar}(s+w)+O_Q(n^{-2})
 \ge1-\frac{C_Qr_n}{n}-\frac{C_Q}{n^2},
\]
because $\partial_yB_{q,\Lstar}(s+w)=O_Q(w)$ and
$0<w\le r_n+1/n$.  The strictly positive excess in
\eqref{eq:exact-critical-barrier} and the mean-value theorem therefore yield,
after one fixed index,
\[
 r_{n+1}\ge r_n
 \left(1-\frac{C_Qr_n}{n}-\frac{C_Q}{n^2}\right).
\]
As $r_n\to0$, taking reciprocals and enlarging the constant gives
\begin{equation}\label{eq:critical-barrier-reciprocal}
 \frac1{r_{n+1}}
 \le\left(1+\frac{C_Q}{n^2}\right)\frac1{r_n}
 +\frac{C_Q}{n}.
\end{equation}
The product $\prod_n(1+C_Q/n^2)$ converges.  Discrete variation of constants
in \eqref{eq:critical-barrier-reciprocal}, together with the positive minimum
of $r_{n_0}(q)$ on $Q$ at a fixed starting index $n_0$, gives
\[
 r_n(q)\ge\frac{c_Q}{1+\log n}.
\]
Consequently $w_n=r_n+1/n\ge c_Q/\ell_n$.  The additive $n^{-2}$ error in
\eqref{eq:critical-riccati} can now be absorbed into $w_n^2/n$.  Uniform
smallness of $w_n$ also absorbs the cubic term, and hence
\begin{equation}\label{eq:critical-upper-drift}
 w_{n+1}\le w_n-\frac{b_-}{2n}w_n^2,
 \qquad b_-:=\min_{q\in Q}b(q)>0.
\end{equation}
Taking reciprocals and summing proves $w_n\le C_Q/\ell_n$.  Thus
$w_n\asymp_Q\ell_n^{-1}$, and a reciprocal expansion in
\eqref{eq:critical-riccati} is now justified term by term:
\[
 \frac1{w_{n+1}}-\frac1{w_n}
 =\frac{b(q)}n
 +O_Q\!\left(\frac1{n\ell_n}+\frac{\ell_n^2}{n^2}\right).
\]
After summation,
$w_n^{-1}=b(q)\log n+O_Q(\log\log n)$; inversion proves
\eqref{eq:riccati-asymptotic}.

\emph{Step 3: the discrete bottleneck passage.}
Put $\delta=\delta_N=\Lstar(q)-\lambda_N(q)$ and
$v_n=y_*-y_n$.  Choose the local radius and constants from
Lemma~\ref{lem:scaled-step}, and write $\sigma_Q=\sigma$.  For $n\ge n_*$
its drift comparison is
\begin{equation}\label{eq:shell-drift}
 \frac{c_Q}{n}(\delta+v_n^2)
 \le v_{n+1}-v_n
 \le\frac{C_Q}{n}(\delta+v_n^2)
\end{equation}
whenever $n^{-1}\le\sigma_Q(\delta+v_n^2)$.  The increments are positive even
before this comparison applies, by Step 1.

Choose a fixed large $n_0\ge n_*$ on the incoming critical orbit.  Step 2 permits
this choice so that $v_{n_0}^*$ lies in a fixed incoming annulus and
$n_0^{-1}\le\sigma_Q|v_{n_0}^*|^2/4$.  Continuity of the first $n_0$
compositions gives the same facts for the subcritical orbit when $\delta$ is
small.  Lemma~\ref{lem:scaled-step} then preserves the condition at every
step throughout the local passage.  Decompose the incoming part into shells
\[
 \rho_{j+1}\le |v|\le\rho_j,
 \qquad \rho_j=2^{-j}\rho_0,
\]
down to $\rho_j\asymp\sqrt\delta$.  On such a shell,
\eqref{eq:shell-drift} shows that its harmonic crossing time $H_j$ satisfies
\begin{equation}\label{eq:shell-time}
 \frac{c_Q}{\rho_j}\le H_j:=\sum_{n\,\text{in the shell}}\frac1n
 \le\frac{C_Q}{\rho_j}.
\end{equation}
The step condition used in this comparison has already been established
independently of \eqref{eq:shell-time}, both within each shell and between
successive shells.  The last-step overshoot is harmless because
\eqref{eq:shell-drift} is $o_Q(\rho_j)$ there.

At $|v|\le C\sqrt\delta$, the same preserved condition applies and
$\delta+v^2\asymp\delta$.  The drift is comparable with $\delta/n$, and the
layer has width comparable with $\sqrt\delta$; its harmonic crossing time is
therefore comparable with $\delta^{-1/2}$.  The outgoing shells obey the same
estimates under the preserved step condition.  Since
$\sum_{\rho_j\ge\sqrt\delta}\rho_j^{-1}\asymp\delta^{-1/2}$, the complete
bottleneck contributes
\begin{equation}\label{eq:bottleneck-coarse-time}
 c_Q\delta^{-1/2}
 \le\sum_{n\,\text{in the bottleneck}}\frac1n
 \le C_Q\delta^{-1/2}.
\end{equation}

\emph{Step 4: entry, exit, and conversion to $N$.}
Before the chosen incoming section and after the outgoing section, the field
is uniformly bounded away from zero on each fixed compact $y$-interval.
The same discrete comparison gives $O_Q(1)$ harmonic time there.  Once
$y\asymp1/n$, Lemma \ref{lem:terminal-layer} gives only $O_Q(1)$ further
indices and hence $O_Q(1)$ additional harmonic time.  Since
\[
 \sum_{n=n_0}^{N-1}\frac1n=\log N+O_Q(1),
\]
\eqref{eq:bottleneck-coarse-time} yields
$\log N\asymp_Q\delta_N^{-1/2}$.  Enlarging the constants to cover the
finitely many remaining values $N\ge2$ proves \eqref{eq:coarse-scale}.
\end{proof}

\begin{lemma}[Uniform terminal layer]\label{lem:terminal-layer}
Let $Q\Subset(-1,\infty)$ and $0<\Lam_-<\Lam_+<\infty$.  For every $A>1$
there are $n_A^{\mathrm{min}}\ge1$ and $\eta_{Q,A}>0$ with the following property.  If
$q\in Q$, $\Lam_-\le\Lam\le\Lam_+$, $n\ge n_A^{\mathrm{min}}$, and an admissible step has
$1<X_n\le A$ and $X_{n+1}\ge1$, then
\begin{equation}\label{eq:terminal-gap}
 X_n-1\ge\eta_{Q,A},
 \qquad X_{n+1}=X_n-1+o_{Q,A}(1).
\end{equation}
Consequently an admissible orbit has only $O_{Q,A}(1)$ consecutive indices in
the layer $1<X_n\le A$.
\end{lemma}

\begin{proof}
Substitution of $y_n=X_n/n$ in the exact map gives, for $q\ne0$,
\begin{equation}\label{eq:X-map}
 X_{n+1}=(X_n-1)(1+h_n)^{1+1/q}
 \left[1+h_n\left(\frac{X_n-1}{n\Lam}\right)^q\right]^{-1/q},
\end{equation}
and at $q=0$,
\begin{equation}\label{eq:X-map-zero}
 X_{n+1}=(X_n-1)(1+h_n)
 \left(\frac{X_n-1}{n\Lam}\right)^{-1/(n+1)}.
\end{equation}
The latter is the removable value of the former.

We first prove the gap in \eqref{eq:terminal-gap}.  Otherwise there are
sequences $n\to\infty$, $q\in Q$, and
$\epsilon=X_n-1\downarrow0$ for which $X_{n+1}\ge1$.  Put
\[
 h=1/n,
 \qquad L=\log\!\left(\frac{n\Lam}{\epsilon}\right),
 \qquad
 \Psi(q,h,L)=
 \frac{\log(1+h)-\log(1+he^{-qL})}{q},
\]
with $\Psi(0,h,L)=hL/(1+h)$.  Then
\begin{equation}\label{eq:terminal-log-form}
 \log X_{n+1}=\log\epsilon+\log(1+h)+\Psi(q,h,L).
\end{equation}
After taking a subsequence, $q\to q_0\in Q$.  If $q_0>0$, then
$0\le\Psi\le hL$, and \eqref{eq:terminal-log-form} tends to $-\infty$.  If
$q_0<0$, write $p=-q\in(0,1)$.  When $he^{pL}\le1$, the last factor in
\eqref{eq:X-map} is bounded and $X_{n+1}\to0$; when $he^{pL}\ge1$,
\begin{equation}\label{eq:terminal-negative-bound}
 X_{n+1}\le n\Lam\left(\frac2n\right)^{1/p}\longrightarrow0,
\end{equation}
because the right-hand side is $O_{Q,\Lam_+}(n^{1-1/p})$,
$1-1/p<0$, and $p$ remains bounded away from $1$.

It remains to make the removable case uniform.  If $q\ge0$ and $q\to0$,
the integral representation
\[
 \Psi=L\int_0^1
 \frac{he^{-\theta qL}}{1+he^{-\theta qL}}\,\dd\theta
\]
again gives $\Psi\le hL$, so \eqref{eq:terminal-log-form} tends to
$-\infty$.  If $q=-p<0$, put $u=pL$.  For $u\ge\log n$,
\eqref{eq:terminal-negative-bound} applies.  For $u\le\log n$,
\begin{equation}\label{eq:terminal-zero-integral}
 \Psi=L\int_0^1
 \frac{he^{\theta u}}{1+he^{\theta u}}\,\dd\theta.
\end{equation}
If $L=O(\log n)$, then $p\to0$ gives $u=o(\log n)$ and the right-hand side
is $O(Ln^{-1+o(1)})=o(1)$.  If $L/\log n\to\infty$, the integrand in
\eqref{eq:terminal-zero-integral} is at most $1/2$, so
$\log X_{n+1}\le\log(n\Lam)-L/2+o(1)\to-\infty$.  These alternatives cover
every subsequence and contradict $X_{n+1}\ge1$.  The gap is proved.

On the compact range
$\eta_{Q,A}\le X_n-1\le A-1$ one has $L=\log n+O_{Q,A}(1)$.  The preceding
integral formulas give, uniformly in $q\in Q$,
\[
 |\Psi(q,h,L)|
 \le C_{Q,A}\bigl(n^{-1}\log n+n^{-1+p_*}\log n\bigr)=o(1),
 \qquad p_*:=\max\{0,-\inf Q\}<1.
\]
Equation \eqref{eq:terminal-log-form} therefore gives the second assertion in
\eqref{eq:terminal-gap}.  Increasing $n_A^{\mathrm{min}}$ makes
$X_n-X_{n+1}\ge1/2$ at every admissible step in the layer, which proves the
uniform bound on its length.
\end{proof}

\section{An abstract discrete phase theorem}
\label{sec:abstract}

We record the analytic mechanism in a form that separates the local
bottleneck from the discrete orbit.  The endpoints are included explicitly,
because in the Hardy application the incoming endpoint moves with the
parameter.  Let $r_0,\delta_0>0$, and let
$\beta_\pm(\delta)$ be holomorphic for $|\delta|<\delta_0$, real for real
$\delta$, with
\[
 \beta_-(\delta)<-2r_0<0<2r_0<\beta_+(\delta)
\]
for small $\delta\ge0$.  Suppose that
\[
 v_1(\delta)=\beta_-(\delta),\qquad
 v_{n+1}=S_{n,\delta}(v_n),
\]
and
\begin{equation}\label{eq:abstract-map-expansion}
 S_{n,\delta}(v)=v+\frac1n g_\delta(v)+R_n(\delta,v),
 \qquad
 |\partial_v^a\partial_\delta^bR_n|\le C_{a,b}n^{-2}.
\end{equation}
Assume explicitly that $g(\delta,v)=g_\delta(v)$ is jointly holomorphic on
a fixed complex neighbourhood of $(0,0)$, and on neighbourhoods of the fixed
compact continuation arcs used below away from any exceptional terminal
endpoint.  Every fixed mixed derivative of $g$ is bounded on smaller compact
subdomains of these charts.  The maps and remainders in
\eqref{eq:abstract-map-expansion} are holomorphic on the corresponding
charts, with the stated derivative bounds uniform in $n$ whenever the chart
is used.  A finite initial prefix is handled in its own regular charts.
The local normal form is
\[
 g_\delta(v)=a\delta+bv^2
 +O(\delta^2+\delta|v|+|v|^3),\qquad a,b>0.
\]
We assume that $g_\delta$ has no real zero on
$[\beta_-(\delta),\beta_+(\delta)]$ for $0<\delta<\delta_0$, and is bounded
away from zero on the two outer parts
$[\beta_-(\delta),-r_0]$ and $[r_0,\beta_+(\delta)]$.
At an exceptional terminal endpoint a continuous extension of $g$ and the
passage integral suffice; holomorphy in $v$ there is not assumed.  Its
contribution is controlled separately by (H5)--(H6).  In particular, the
Hardy endpoint $y=0$ is not included in the holomorphic charts for
\eqref{eq:abstract-map-expansion} or (H4).
Let $\ell_n=1+\log n$.  For $\delta>0$, define
\begin{align}
 I(\delta)&=\int_{\beta_-(\delta)}^{\beta_+(\delta)}
 \frac{\dd v}{g_\delta(v)},\label{eq:I-def}\\
 F_\delta(v)&=\int_{\beta_-(\delta)}^{v}
 \frac{\dd s}{g_\delta(s)},\label{eq:F-def}\\
 d_{n,\delta}(v)&=\log(1+1/n)
 -\int_v^{S_{n,\delta}(v)}\frac{\dd s}{g_\delta(s)}.
 \label{eq:defect}
\end{align}
Since $F_\delta(v_1(\delta))=0$, the phase defect obeys the exact telescoping
identity
\begin{equation}\label{eq:telescoping}
 \log m-F_\delta(v_m)
 =\sum_{n=1}^{m-1}d_{n,\delta}(v_n).
\end{equation}

We will repeatedly use the following elementary moving-endpoint fact.  If
$a(\delta),b(\delta)$ and $\phi(\delta,v)$ are holomorphic and $\phi$ is
holomorphic on a neighbourhood of the intervening arcs, then
\begin{equation}\label{eq:moving-endpoint-rule}
 G(\delta)=\int_{a(\delta)}^{b(\delta)}\phi(\delta,v)\,\dd v
\end{equation}
is holomorphic.  This follows by parametrizing the segment by
$v=a(\delta)+s(b(\delta)-a(\delta))$, $0\le s\le1$; differentiation gives,
for example,
\[
 G'(\delta)=b'(\delta)\phi(\delta,b(\delta))
 -a'(\delta)\phi(\delta,a(\delta))
 +\int_{a(\delta)}^{b(\delta)}\partial_\delta\phi(\delta,v)\,\dd v.
\]
Consequently movement of an endpoint through a regular holomorphic chart
where $g_\delta$ does not vanish contributes a function holomorphic in
$\delta$.  This rule is not applied at an exceptional terminal endpoint;
any parameter regularity there must be justified separately.  In the
variable $\xi=\sqrt\delta$ it changes only the regular even part of the
passage-time expansion.  It can change $J,\mu_2,\mu_4,\ldots$, but it creates
neither the pole in $\xi$ nor logarithmic terms.  Notice also that the moving
lower endpoint causes no extra term in \eqref{eq:telescoping}: it is exactly
the initial condition and hence $F_\delta(v_1)=0$ for every $\delta$.

We use the following six hypotheses; the parameter-uniform versions are
understood when a compact parameter set is present.

\begin{enumerate}
 \item[(H1)] The terminal problem selects a unique $\delta_N>0$, with
 $\delta_N\to0$ and $\delta_N(\log N)^2\asymp1$.
 \item[(H2)] The critical orbit remains on the incoming side and
 \[
  -v_n^*=\frac1{b\log n}
  +O\!\left(\frac{\log\log n}{\log^2n}\right).
 \]
 \item[(H3)] For real $\delta$ in the wedge
 $0\le\delta\ell_m^2\le\varepsilon_0$, and for
 $n\le m$, the map is analytic, the orbit satisfies
 \[
  \rho_n^{-1}\le C\ell_n^2,
  \qquad \rho_n=\delta+v_n(\delta)^2,
 \]
 and its variational propagator contracts as
 \[
  1-\frac{c_2}{n\ell_n}
  \le\partial_vS_{n,\delta}(v_n(\delta))
  \le1-\frac{c_1}{n\ell_n}.
 \]
 Here $c_2\ge c_1>0$.
 The derivative bounds in \eqref{eq:abstract-map-expansion} hold for every
 fixed pair $(a,b)$, uniformly in the same wedge.  A fixed initial prefix may
 be absorbed into the constants.
 \item[(H4)] If $\rho=\delta+v^2$ and
 $n^{-1}\le\sigma_0\rho$ for a fixed $\sigma_0>0$, then
 \begin{equation}\label{eq:H4}
  |\partial_v^a\partial_\delta^b d_{n,\delta}(v)|
  \le \frac{C_{a,b}}{n^2\rho^{1+a/2+b}}.
 \end{equation}
 \item[(H5)] From a bottleneck index $m$ satisfying the scaled-step condition
 $m^{-1}\le\sigma_0(\delta+v_m^2)$ to the exact terminal condition,
 \begin{equation}\label{eq:H5}
 \left|\bigl(\log N-I(\delta)\bigr)
 -\bigl(\log m-F_\delta(v_m)\bigr)\right|
 \le C\left(\frac1{m\delta}+\frac{\ell_m}{m}+N^{-c_*}\right).
 \end{equation}
 \item[(H6)] For $\xi=\sqrt\delta$ the continuous passage time, with the
 moving endpoints in \eqref{eq:I-def}, has in a fixed sector
 \begin{equation}\label{eq:H6}
 I(\xi^2)=\frac{c}{\xi}+J+\sum_{k=1}^M\mu_k\xi^k
 +O_M(\xi^{M+1}),\qquad c=\frac{\pi}{\sqrt{ab}},
 \end{equation}
 at every fixed order, with no $\xi^k\log\xi$ terms.
\end{enumerate}

\begin{lemma}[Discrete Fa\`a di Bruno estimate]\label{lem:discrete-jet}
Under \emph{(H3)--(H4)}, for every fixed $j\ge1$ and every
$n\le m$ at a real point of the wedge
$0\le\delta\ell_m^2\le\varepsilon_0$,
\begin{align}
 \left|\frac{\dd^jv_n}{\dd\delta^j}\right|
 &\le C_j\ell_n^{2j-1},\label{eq:jet-lemma-orbit}\\
 \left|\frac{\dd^j}{\dd\delta^j}
 d_{n,\delta}(v_n(\delta))\right|
 &\le C_jn^{-2}\ell_n^{2j+2}.
 \label{eq:jet-lemma-defect}
\end{align}
The constants may depend on $j$; no estimate uniform in the jet order is
asserted.
\end{lemma}

\begin{proof}
Set $u_n^{(r)}=\dd^rv_n/\dd\delta^r$ and
$D_n=\partial_vS_{n,\delta}(v_n)$.  The product estimate in (H3) is
 \begin{equation}\label{eq:log-propagator}
 \mathcal P(n,s):=\prod_{k=s}^{n-1}D_k,
 \qquad
 |\mathcal P(n,s)|\le C\left(\frac{\ell_s}{\ell_n}\right)^{c_1},
 \qquad 1\le s<n.
\end{equation}
Indeed, after one fixed prefix the bounds in (H3) give $0<D_k<1$; that prefix
is absorbed into $C$.  Take logarithms, use $\log(1-x)\le-x$, and compare
$\sum_{k=s}^{n-1}(k\ell_k)^{-1}$ with
$\log(\ell_n/\ell_s)$.

For clarity, introduce the finite partition set
\[
 \mathfrak P_j=
 \left\{(b,k_1,\ldots,k_j)\in\mathbb Z_{\ge0}^{j+1}:
 b+\sum_{r=1}^j r k_r=j\right\},
 \qquad a(k)=\sum_{r=1}^jk_r.
\]
The one-variable multivariate Fa\`a di Bruno formula reads
\begin{equation}\label{eq:faa-explicit}
 \frac{\dd^j}{\dd\delta^j}S_{n,\delta}(v_n(\delta))
 =\sum_{\mathfrak P_j}
 \frac{j!}{b!\prod_{r=1}^jk_r!(r!)^{k_r}}
 (\partial_\delta^b\partial_v^{a(k)}S_{n,\delta})(v_n)
 \prod_{r=1}^j(u_n^{(r)})^{k_r}.
\end{equation}
The unique summand containing $u_n^{(j)}$ is
$D_nu_n^{(j)}$, corresponding to $b=0,k_j=1$.  In every other summand,
\eqref{eq:abstract-map-expansion} and analyticity of $g$ give
\[
 |\partial_\delta^b\partial_v^{a(k)}S_{n,\delta}|
 \le C_{j}/n.
\]
Assume inductively that \eqref{eq:jet-lemma-orbit} holds through order
$j-1$.  The product of lower jets in such a summand is bounded by
\begin{equation}\label{eq:faa-power-count}
 C_j\ell_n^{\sum_r(2r-1)k_r}
 =C_j\ell_n^{2(j-b)-a(k)}.
\end{equation}
For every nonprincipal partition one has $2b+a(k)\ge2$: the only partition
with $b=0$ and $a(k)=1$ is the principal one already removed.  Hence the
forcing $E_{n,j}$ in
\[
 u_{n+1}^{(j)}=D_nu_n^{(j)}+E_{n,j}
\]
satisfies
\begin{equation}\label{eq:jet-forcing}
 |E_{n,j}|\le \frac{C_j}{n}\ell_n^{2j-2}.
\end{equation}
All derivatives of $v_1=\beta_-(\delta)$ are bounded.  Discrete variation of
constants, \eqref{eq:log-propagator}, and the integral test now give
\begin{align*}
 |u_n^{(j)}|
 &\le C_j+C_j\sum_{s<n}
 \left(\frac{\ell_s}{\ell_n}\right)^{c_1}
 \frac{\ell_s^{2j-2}}s\\
 &\le C_j\ell_n^{-c_1}
 \left(1+\int_1^{\ell_n}u^{2j-2+c_1}\,\dd u\right)
 \le C_j\ell_n^{2j-1}.
\end{align*}
This closes the induction and proves \eqref{eq:jet-lemma-orbit}.

By (H3), $\rho_n\ge C^{-1}\ell_n^{-2}$, and hence
$n^{-1}\le\sigma_0\rho_n$ after one fixed index.  The finitely many preceding
defect terms have bounded jets and are absorbed into the constants.  We may
therefore apply the same partition formula and (H4) to
$d_{n,\delta}(v_n(\delta))$.  For a term indexed by
$(b,k_1,\ldots,k_j)$, (H4) and
$\rho_n^{-1}\le C\ell_n^2$ give
\[
 |\partial_\delta^b\partial_v^{a(k)}d_{n,\delta}(v_n)|
 \le C_jn^{-2}\ell_n^{2+a(k)+2b}.
\]
Multiplication by \eqref{eq:faa-power-count} produces exactly
$n^{-2}\ell_n^{2j+2}$.  The set $\mathfrak P_j$ is finite for fixed $j$,
so summing its terms proves \eqref{eq:jet-lemma-defect}.
\end{proof}

\begin{theorem}[Abstract phase theorem]\label{thm:abstract}
Under \emph{(H1)--(H6)}, for every fixed $j\ge1$, every $n\le m$, and every
real $\delta\ge0$ in the wedge $\delta\ell_m^2\le\varepsilon_0$,
\begin{align}
 |\partial_\delta^jv_n(\delta)|
 &\le C_j\ell_n^{2j-1},\label{eq:orbit-jets}\\
 \left|\frac{\dd^j}{\dd\delta^j}
 d_{n,\delta}(v_n(\delta))\right|
 &\le C_j\frac{\ell_n^{2j+2}}{n^2}.
 \label{eq:defect-jets}
\end{align}
Consequently the absolutely convergent coefficients
\begin{equation}\label{eq:tau}
 \tau_j=\frac1{j!}\sum_{n=1}^{\infty}
 \left.\frac{\dd^j}{\dd\delta^j}
 d_{n,\delta}(v_n(\delta))\right|_{\delta=0}
\end{equation}
satisfy, at every fixed order $L$,
\begin{equation}\label{eq:phase-equation}
 \log N-I(\delta_N)
 =\sum_{j=0}^{L}\tau_j\delta_N^j+O_L(\delta_N^{L+1}).
\end{equation}
With
\[
 \kappa=-(J+\tau_0),\qquad X=\log N+\kappa,
\]
there is a unique Poincar\'e expansion
\begin{equation}\label{eq:abstract-expansion}
 \delta_N\sim\sum_{j\ge2}A_jX^{-j},
\end{equation}
where
\begin{equation}\label{eq:abstract-coefficients}
 A_2=c^2=\frac{\pi^2}{ab},\qquad A_3=0,
 \qquad A_4=2\mu_1c^3,
 \qquad A_5=2(\mu_2+\tau_1)c^4.
\end{equation}
\end{theorem}

\begin{proof}
Lemma \ref{lem:discrete-jet} proves
\eqref{eq:orbit-jets}--\eqref{eq:defect-jets}.  At $\delta=0$ the latter is
summable in $n$ for $j\ge1$.  For $j=0$, (H3)--(H4) give directly
$|d_{n,0}(v_n^*)|\le Cn^{-2}\ell_n^2$ after a fixed prefix.  Hence every
series in \eqref{eq:tau}, including $\tau_0$, converges absolutely.
All derivatives created by the moving lower endpoint occur in the bounded
initial jets of $v_1=\beta_-(\delta)$; the moving contribution to $I$ is the
holomorphic function described by \eqref{eq:moving-endpoint-rule}.  Thus no
unrecorded Leibniz boundary term remains in the truncation argument.

For a fixed $L$, choose
\[
 m_L=\lfloor\delta_N^{-(L+2)}\rfloor.
\]
By (H1), $m_L<N$ and $\delta_N\ell_{m_L}^2\to0$.  The estimate
\eqref{eq:orbit-jets} with $j=1$, integrated from $0$ to
$\delta_N$, and (H2) also give
\[
 v_{m_L}(\delta_N)=v_{m_L}^*+O(\delta_N\ell_{m_L})
 =-\frac{1+o(1)}{b\log m_L}.
\]
Thus $m_L$ is on the incoming bottleneck section for all large $N$, and
$m_L^{-1}\le\sigma_0(\delta_N+v_{m_L}^2)$; in particular (H5) is applicable
at this index.  Taylor-expand the finite
sum in \eqref{eq:telescoping} through order $L$.  Formula
\eqref{eq:defect-jets} makes the total Taylor remainder
$O(\delta_N^{L+1})$.  Replacing its finite coefficients by the series
\eqref{eq:tau} costs
\[
 O_L\!\left(\sum_{j=0}^L
 \delta_N^j\frac{\ell_{m_L}^{2j+2}}{m_L}\right)
 =o(\delta_N^{L+1}).
\]
The largest terminal error in (H5) is
$(m_L\delta_N)^{-1}=O(\delta_N^{L+1})$, and $N^{-c_*}$ is smaller than every
fixed power of $\delta_N$.  This proves \eqref{eq:phase-equation}.

Insert (H6), put $\xi_N=\sqrt{\delta_N}$, and absorb the constant term into
the phase.  More generally, at every fixed order this gives
\[
 X=\frac{c}{\xi_N}+\sum_{k=1}^{M}\nu_k\xi_N^k
 +O_M(\xi_N^{M+1}),
 \qquad
 \nu_k=
 \begin{cases}
  \mu_k+\tau_{k/2},&k\ge2\text{ even},\\
  \mu_k,&k\text{ odd}.
 \end{cases}
\]
In particular,
\begin{equation}\label{eq:inverse-equation}
 X=\frac{c}{\xi_N}+\mu_1\xi_N
 +(\mu_2+\tau_1)\xi_N^2+O(\xi_N^3).
\end{equation}
Put $z=X^{-1}$.  The leading term and (H1) first give
$\xi_N=cz+O(z^3)$.  Multiplying a fixed-order relation by $z\xi_N$ rewrites it
as
\begin{equation}\label{eq:asymptotic-reversion}
 \xi_N=cz+z\sum_{k=1}^{M}\nu_k\xi_N^{k+1}
 +O_M(z\xi_N^{M+2}).
\end{equation}
This identity permits finite-order asymptotic reversion without any convergence
assumption.  For a fixed $r$, choose a polynomial
$P_r(z)=cz+\sum_{m=2}^ra_mz^m$ recursively so that
\[
 P_r-cz-z\sum_{k=1}^{r}\nu_kP_r^{k+1}=O(z^{r+1}).
\]
At degree $m$, the coefficient $a_m$ occurs only on the left; the coefficient
of degree $m$ on the right depends solely on $a_1,\ldots,a_{m-1}$.  Hence this
recursion exists and is unique.  Use the phase relation to an order larger than
$r$ and subtract the displayed identity from
\eqref{eq:asymptotic-reversion}.  Since $\xi_N,P_r=O(z)$,
\[
 |\xi_N-P_r|
 \le C_rz^2|\xi_N-P_r|+O_r(z^{r+1}).
\]
The first term is absorbed for small $z$, uniformly on a compact parameter set,
so $\xi_N-P_r=O_r(z^{r+1})$.  Thus $\xi_N$ has a unique fixed-order Poincar\'e
expansion.  Its first terms are
\[
 \xi_N=cz+\mu_1c^2z^3+(\mu_2+\tau_1)c^3z^4+O(z^5).
\]
Squaring gives
\eqref{eq:abstract-expansion}--\eqref{eq:abstract-coefficients}.  In
particular, the phase normalization removes the cubic term.  The coefficient
$A_4$ uses only local continuous-flow data, whereas $A_5$ is the first coefficient
that reads the global discrete response $\tau_1$.
Moreover, replacing $\kappa$ by $\kappa+s$ changes the coefficient of
$X^{-3}$ by $2sA_2$.  Since $A_2=c^2>0$, the condition $A_3=0$ fixes the phase
uniquely.
\end{proof}

\section{Verification for the power-mean family}
\label{sec:verification}

We now verify the hypotheses uniformly on $Q\Subset(-1,\infty)$.  Lemma
\ref{lem:capture} provides (H1)--(H2).  The remaining ingredients are stated
in a form that also records the parameter dependence required by Theorem
\ref{thm:main}.

In the coordinate $v=y_*(q)-y$, the abstract endpoints are explicitly
\begin{equation}\label{eq:hardy-moving-endpoints}
 \beta_-(q,\delta)=y_*(q)-\Lstar(q)+\delta,
 \qquad
 \beta_+(q,\delta)=y_*(q).
\end{equation}
The first one is the initial value $v_1$, while the second corresponds to
$y=0$.  Consequently
\[
 \int_{\beta_-(q,\delta)}^{\beta_+(q,\delta)}
 \frac{\dd v}{g_{q,\delta}(v)}
 =\int_0^{\Lstar(q)-\delta}
 \frac{\dd y}{B_{q,\Lstar(q)-\delta}(y)}.
\]
Since $B_{q,\Lam}(\Lam)=1$, the moving incoming endpoint stays uniformly
regular on compact $q$-sets, so the moving-boundary rule
\eqref{eq:moving-endpoint-rule} applies there.  At the fixed endpoint $y=0$
we do not assume holomorphy in $y$; the terminal transfer and parameter
regularity of its integral are proved separately in Propositions
\ref{prop:terminal} and \ref{prop:nolog}.

\begin{proposition}[Uniform analytic wedge and phase jets]\label{prop:wedge}
For every $Q\Subset(-1,\infty)$, the translated exact map associated with
\eqref{eq:exact-map} has a common holomorphic chart near the bottleneck in
$(q,\delta,v,1/n)$.  It satisfies (H3) on the real incoming wedge and (H4)
at real centres in the analytic passage region, with the complex extensions
specified in the proof.  For real $q\in Q$, real
$0\le\delta\ell_n^2\le\varepsilon_0$, and every fixed $j,m\ge0$,
\begin{equation}\label{eq:q-orbit-jets}
 |\partial_q^m v_n(q,0)|\le \frac{C_{Q,m}}{\ell_n},
 \qquad
 |\partial_q^m\partial_\delta^jv_n(q,\delta)|
 \le C_{Q,j,m}\ell_n^{2j-1}\quad(j\ge1),
\end{equation}
and
\begin{equation}\label{eq:q-defect-jets}
 \left|\partial_q^m\partial_\delta^j
 d_{n,q,\delta}(v_n(q,\delta))\right|
 \le C_{Q,j,m}\frac{\ell_n^{2j+2}}{n^2}.
\end{equation}
Hence every phase coefficient $\tau_j(q)$ is holomorphic in a complex
neighbourhood of $Q$ and real analytic on $Q$.
\end{proposition}

\begin{proof}
We give the analytic-domain, wedge, defect, and mixed-jet arguments
separately.

\emph{Step 1: a common holomorphic chart.}
With $h$ in place of $1/n$ and $L=\log((y-h)/\Lam)$, the nonalgebraic factor
in \eqref{eq:exact-map} is
\[
 \exp\Theta(q,L,h),
 \qquad
 \Theta(q,L,h)=
 \frac{\log(1+h)-\log(1+he^{qL})}{q}.
\]
The identity
\begin{equation}\label{eq:theta-removable}
 \Theta(q,L,h)
 =-L\int_0^1\frac{he^{\theta qL}}{1+he^{\theta qL}}\,\dd\theta
\end{equation}
removes $q=0$ and fixes all logarithm branches near the positive real data.
Write
\[
 \Lstar(q)=\exp\left(\frac{1+q}{q}\log(1+q)\right).
\]
Because $Q$ has positive distance from $-1$, the principal logarithm and
\eqref{eq:theta-removable} are holomorphic on one complex tube
$\mathcal U_Q$ around $Q$.  After fixing small radii in $(\delta,v,h)$, the
quantities $y_*(q)-v-h$ and $\Lstar(q)-\delta$ stay in compact subsets of the
same logarithm domain.  Hence
\[
 \mathcal S(q,\delta,v,h)
 =y_*(q)-T_{h,q,\Lstar(q)-\delta}(y_*(q)-v)
\]
is jointly holomorphic on a common polydisc.  Taylor's formula in $h$ gives
\begin{equation}\label{eq:exact-map-remainder}
 \mathcal S(q,\delta,v,h)
 =v+h g_{q,\delta}(v)+h^2\mathcal R(q,\delta,v,h),
\end{equation}
where every fixed mixed derivative of $\mathcal R$ is bounded uniformly on a
smaller common polydisc.  This proves the analytic remainder assertion in
(H3), including its parameter-uniform meaning.

\emph{Step 2: real and complex wedge bootstrap.}
On the critical orbit, Lemma \ref{lem:capture} and
\eqref{eq:normal-form} give
\begin{equation}\label{eq:critical-derivative}
 D_n^*:=\partial_v\mathcal S(q,0,v_n^*,1/n)
 =1-\frac{2}{n\ell_n}
 +O_Q\!\left(\frac{\log\ell_n}{n\ell_n^2}+\frac1{n^2}\right).
\end{equation}
Fix $1<\gamma<2$.  After one uniform initial index,
\begin{equation}\label{eq:complex-propagator}
 \left|\prod_{k=s}^{n-1}D_k^*\right|
 \le C_Q\left(\frac{\ell_s}{\ell_n}\right)^\gamma,
 \qquad n>s,
\end{equation}
and the real derivatives satisfy the two-sided contraction in (H3).

These bounds persist on a genuine complex $q$-tube.  To see this without an
implicit compactness assumption, fix $q_0\in Q$ and compare the critical
recurrences at $q$ and $q_0$.  In the moving coordinate,
\begin{equation}\label{eq:q-flatness}
 \partial_q^r g_{q,0}(v)=O_{Q,r}(v^2),
 \qquad
 \partial_q^r\partial_vg_{q,0}(v)=O_{Q,r}(|v|)\qquad(r\ge1),
\end{equation}
because the constant and linear critical coefficients vanish identically in
$q$.  Thus the difference equation has forcing
$O_Q(|q-q_0|/(n\ell_n^2)+|q-q_0|/n^2)$.  Variation of constants with
\eqref{eq:complex-propagator} gives
\[
 |v_n(q,0)-v_n(q_0,0)|
 \le C_Q|q-q_0|/\ell_n.
\]
A first-failure argument justifies the estimate on a sufficiently small disc
about $q_0$; finitely many such discs cover $Q$.  This already gives
$|v_n(q,0)|\asymp_Q\ell_n^{-1}$ on a common complex tube.  Subtracting the
leading term in the complex critical Riccati recurrence and repeating the
reciprocal summation from Lemma \ref{lem:capture} then gives
\[
 v_n(q,0)=-\frac1{b(q)\ell_n}
 +O_Q\left(\frac{\log\ell_n}{\ell_n^2}\right)
\]
there as well.

Now put $p_n=v_n(q,\delta)-v_n(q,0)$.  Expansion of
\eqref{eq:exact-map-remainder} about the critical orbit yields
\begin{equation}\label{eq:wedge-perturbation}
 |p_{n+1}-D_n^*p_n|
 \le \frac{C_Q}{n}
 \bigl(|\delta|+|p_n|^2+|\delta p_n|+|\delta|^2\bigr)
 +\frac{C_Q}{n^2}(|p_n|+|\delta|).
\end{equation}
The initial displacement is $O_Q(\delta)$.  Assuming
$|p_k|\le M|\delta|\ell_k$ up to a first possible failure, using
$|\delta|\ell_m^2\le\varepsilon_0$, and applying
\eqref{eq:complex-propagator} gives
\[
 |p_n|\le C_Q(1+M^2\varepsilon_0)|\delta|\ell_n,
 \qquad n\le m.
\]
Choose $M$ first and then $\varepsilon_0$ sufficiently small.  The strict
bootstrap improvement closes the estimate for real or complex $\delta$.
After shrinking the common tube once more,
\begin{equation}\label{eq:wedge-separation}
 |v_n(q,\delta)|\asymp_Q\ell_n^{-1},
 \qquad
 \widehat\rho_n:=|\delta|+|v_n(q,\delta)|^2\asymp_Q\ell_n^{-2}
 \quad(n\le m).
\end{equation}
The corresponding variational coefficient differs from $D_n^*$ by at most
$C_Q\varepsilon_0/(n\ell_n)+C_Q/n^2$.  After decreasing
$\varepsilon_0$, the complex product estimate
\eqref{eq:complex-propagator} therefore remains valid along the whole wedge,
with a possibly smaller exponent still larger than $1$.
For real $q$ and $\delta$, differentiating
\eqref{eq:exact-map-remainder} and using
\eqref{eq:wedge-separation} preserves both inequalities in the contraction
part of (H3).

\emph{Step 3: a zero-free scaled polydisc and (H4).}
We prove the scaled estimate at every real centre in the bottleneck, not only
at the incoming wedge points from Step~2.  After fixing a sufficiently small
real neighbourhood, the normal form and the positivity of $a(q),b(q)$ give,
uniformly for $q\in Q$, $\delta\ge0$, and
$\rho:=\delta+v^2>0$,
\begin{equation}\label{eq:full-bottleneck-field}
 c_Q\rho\le g_{q,\delta}(v)\le C_Q\rho,
 \qquad
 |\partial_vg_{q,\delta}(v)|\le C_Q\sqrt\rho,
 \qquad
 |\partial_\delta g_{q,\delta}(v)|\le C_Q.
\end{equation}
Moreover, for every fixed $r\ge1$,
\begin{equation}\label{eq:full-bottleneck-q-flatness}
 |\partial_q^r g_{q,\delta}(v)|\le C_{Q,r}\rho.
\end{equation}
Indeed, at $\delta=0$ the constant and linear terms vanish identically in
$q$, while differentiation of the coefficients of $\delta$ and $v^2$ in
\eqref{eq:normal-form} preserves their orders.  The error terms are absorbed
by decreasing the real neighbourhood once.

We next complexify these bounds about an arbitrary such real centre.  Choose
$c_0>0$ smaller than the width of the common $q$-tube and consider
\begin{equation}\label{eq:defect-polydisc}
 |z-v|\le c_0\sqrt\rho,\qquad
 |\delta'-\delta|\le c_0\rho,\qquad
 |q'-q|\le c_0
\end{equation}
and a concentric polydisc with all three radii enlarged by a fixed factor.
For points in the enlarged polydisc one has
$|z|\le C_Q\sqrt\rho$ and $|\delta'|\le C_Q\rho$.  Analytic Taylor estimates,
\eqref{eq:full-bottleneck-field}, and
\eqref{eq:full-bottleneck-q-flatness} therefore give
\[
 |g_{q',\delta'}(z)-g_{q,\delta}(v)|\le C_Qc_0\rho.
\]
Taking $c_0$ sufficiently small yields
\begin{equation}\label{eq:full-bottleneck-zero-free}
 |g_{q',\delta'}(z)|\ge c'_Q\rho
\end{equation}
throughout the smaller polydisc.  This includes the centre $v=0$, $\delta>0$.
The $q'$-radius is independent of $\rho$ precisely because
$\partial_q^r g=O_Q(\delta+v^2)$ in the moving critical coordinate; only the
transverse parameter $\delta'$ has radius $O(\rho)$.  When $\delta=0$, the
choice of $c_0$ also keeps the $z$-disc a fixed relative distance from the
double zero at $z=0$ whenever its real centre is nonzero.

Now let $h\le\sigma_0\rho$, with $\sigma_0$ decreased if necessary.
Equation \eqref{eq:exact-map-remainder} and $h^2\le\sigma_0h\rho$ give
\[
 \Delta:=\mathcal S(q,\delta,v,h)-v
 =hg_{q,\delta}(v)+O_Q(h^2),
 \qquad |\Delta|\le C_Qh\rho.
\]
The same estimates hold at complex points in a smaller polydisc, and the
segment joining the point to its image remains inside the zero-free chart.
There $1/g$ has a single-valued holomorphic primitive, so
\[
 d_{h,q,\delta}(v)=\log(1+h)
 -\Delta\int_0^1\frac{\dd\theta}
 {g_{q,\delta}(v+\theta\Delta)}
\]
there.  Uniform boundedness of the analytic remainder gives
\[
 \frac{\Delta}{g_{q,\delta}(v)}
 =h+O_Q\!\left(\frac{h^2}{\rho}\right).
\]
Moreover, the complex versions of \eqref{eq:full-bottleneck-field} and
\eqref{eq:full-bottleneck-zero-free} imply
\[
 \left|
 \Delta\int_0^1\frac{\dd\theta}
 {g_{q,\delta}(v+\theta\Delta)}
 -\frac{\Delta}{g_{q,\delta}(v)}
 \right|
 \le C_Q|\Delta|^2\frac{\sqrt\rho}{\rho^2}
 \le C_Qh^2\sqrt\rho.
\]
Together with \(\log(1+h)=h+O(h^2)\), this exhibits the first-order
cancellation and proves
\[
 |d_{h,q,\delta}(v)|\le C_Qh^2/\rho
\]
throughout that smaller polydisc.
Cauchy's estimates on the nested polydiscs yield, for every fixed $r,a,b$,
\begin{equation}\label{eq:full-defect-cauchy}
 |\partial_q^r\partial_v^a\partial_\delta^b
 d_{h,q,\delta}(v)|
 \le C_{Q,r,a,b}\frac{h^2}{\rho^{1+a/2+b}}.
\end{equation}
On each fixed compact continuation outside the local bottleneck and still
away from the small-$y$ terminal endpoint, $g$ is bounded away from zero and
the same estimate follows directly by compactness (there $\rho$ is bounded
above and below).  Thus \eqref{eq:full-defect-cauchy} proves (H4) throughout
the analytic passage region where it is invoked, including the inner and
outgoing bottleneck, and also records the uniform $q$-derivatives used below.
The separate small-$y$ terminal estimate is proved in Proposition
\ref{prop:terminal}.  Along the incoming wedge the scaled-step condition
follows, after a fixed prefix, from \eqref{eq:wedge-separation}.

\emph{Step 4: mixed Fa\`a di Bruno induction.}
Let
\[
 U_n^{j,m}=\partial_\delta^j\partial_q^m v_n(q,\delta),
 \qquad \omega_j=2j-1\quad(j\ge0).
\]
We prove $|U_n^{j,m}|\le C_{Q,j,m}\ell_n^{\omega_j}$ lexicographically in
$j$ and then $m$.  Besides the principal coefficient
$D_n=\partial_v\mathcal S(q,\delta,v_n,1/n)$, the normal form and
\eqref{eq:exact-map-remainder} give along the wedge, for $r\ge1$,
\begin{align}
 |\partial_q^r\mathcal S|&\le
 C_{Q,r}\left(\frac1{n\ell_n^2}+\frac1{n^2}\right),
 \label{eq:mixed-map-zero}\\
 |\partial_q^r\partial_v\mathcal S|&\le
 C_{Q,r}\left(\frac1{n\ell_n}+\frac1{n^2}\right),
 \label{eq:mixed-map-one}\\
 |\partial_q^r\partial_\delta^b\partial_v^a\mathcal S|
 &\le C_{Q,r,a,b}/n\qquad(2b+a\ge2).
 \label{eq:mixed-map-higher}
\end{align}
The estimates include $r=0$ in \eqref{eq:mixed-map-higher}.

Apply the finite multivariate partition formula to the recurrence.  A term
with $b$ explicit $\delta$-derivatives on the map and $a$ differentiated
orbit factors has $\delta$-orders $j_1,\ldots,j_a$ satisfying
$b+\sum j_i=j$.  The induction hypothesis assigns their product the exact
logarithmic weight
\begin{equation}\label{eq:mixed-power-count}
 \sum_{i=1}^a(2j_i-1)=2(j-b)-a.
\end{equation}
If $2b+a\ge2$, \eqref{eq:mixed-map-higher} makes the resulting forcing at
most $C\ell_n^{2j-2}/n$.  If $b=0,a=1$ but the term is not principal, at
least one $q$-derivative hits the coefficient or a lower-$m$ factor;
\eqref{eq:mixed-map-one} gives the same bound.  The remaining case
$a=b=0$ can occur only for a pure $q$-derivative and is bounded by
\eqref{eq:mixed-map-zero}.  Therefore
\[
 U_{n+1}^{j,m}=D_nU_n^{j,m}+E_n^{j,m},
 \qquad
 |E_n^{j,m}|\le
 \begin{cases}
 C_{Q,m}/(n\ell_n^2),&j=0,\\
 C_{Q,j,m}\ell_n^{2j-2}/n,&j\ge1.
 \end{cases}
\]
All initial jets are bounded.  Variation of constants with exponent
$\gamma>1$ in \eqref{eq:complex-propagator} gives respectively
$O_Q(\ell_n^{-1})$ and $O_Q(\ell_n^{2j-1})$.  This closes the induction and
proves \eqref{eq:q-orbit-jets}; in particular, no $q$-derivative costs an
extra power of $\ell_n$.

Finally apply the same multivariate partition formula to
$d_{n,q,\delta}(v_n)$.  From \eqref{eq:full-defect-cauchy} and
\eqref{eq:wedge-separation}, a term with indices $(a,b)$ contributes
$n^{-2}\ell_n^{2+a+2b}$.  Multiplication by
\eqref{eq:mixed-power-count} gives exactly
$n^{-2}\ell_n^{2j+2}$, independently of how the $q$-derivatives are
distributed.  This proves \eqref{eq:q-defect-jets}.

To justify normal convergence in the complex parameter, fix a smaller open
tube $\mathcal V_Q$ with
$Q\subset\mathcal V_Q\Subset\mathcal U_Q$, on which the complex bootstrap
of Step~2 and the nested charts of Step~3 hold, including the finite initial
prefix.  For every compact $K\Subset\mathcal V_Q$, these estimates give
$\varepsilon_K,C_K>0$, independent of $n$, such that
$d_{n,q,\delta}(v_n(q,\delta))$ is holomorphic for $q$ near $K$ and
$|\delta|<2\varepsilon_K\ell_n^{-2}$, and
\[
 \sup_{\substack{q\in K\\ |\delta|\le\varepsilon_K\ell_n^{-2}}}
 |d_{n,q,\delta}(v_n(q,\delta))|
 \le C_K n^{-2}\ell_n^2.
\]
Indeed, the complex orbit stays within a fixed small relative distance of
the real incoming centre in the scaled charts of Step~3; shrinking the tube
and $\varepsilon_K$ ensures this simultaneously for all large $n$.
The finitely many earlier steps lie in regular charts and are absorbed into
$C_K$.  Cauchy's estimate in the $\delta$-disc therefore yields, for each
fixed $j\ge0$,
\begin{equation}\label{eq:complex-normal-majorant}
 \sup_{q\in K}\left|
 \left.\frac{\dd^j}{\dd\delta^j}
 d_{n,q,\delta}(v_n(q,\delta))\right|_{\delta=0}\right|
 \le C_{K,j}n^{-2}\ell_n^{2j+2}.
\end{equation}
The right-hand side is summable in $n$.  Thus the holomorphic summands
defining $\tau_j$ converge uniformly on every such $K$, and the Weierstrass
theorem makes $\tau_j$ holomorphic on $\mathcal V_Q$.
\end{proof}
\medskip
\begin{proposition}[Uniform terminal transfer]\label{prop:terminal}
The recurrences \eqref{eq:X-map}--\eqref{eq:X-map-zero} are jointly analytic
in $q$ through $q=0$.  For every fixed $A>1$, an admissible orbit has only
$O_{Q,A}(1)$ indices in the terminal layer $1<X_n<A$.  From any bottleneck
index $m$ satisfying the scaled-step condition in (H5), the terminal
comparison holds uniformly for $q\in Q$; in fact one may take $c_*=1$ there.
\end{proposition}

\begin{proof}
Analyticity is already explicit in \eqref{eq:theta-removable}, and the
terminal-layer count is Lemma \ref{lem:terminal-layer}.  It remains to prove
the quantitative transfer to the continuous passage time.

Let $\delta=\Lstar(q)-\Lam>0$ and define the one-step error
\begin{equation}\label{eq:terminal-step-defect}
 \epsilon_n=\log(1+1/n)
 -\int_{y_{n+1}}^{y_n}\frac{\dd y}{B_{q,\Lam}(y)}.
\end{equation}
Here $y_{n+1}<y_n$, the integrand is positive, and the integral is the
continuous one-step passage time appearing in \eqref{eq:defect}.  This order
of the endpoints is essential.
Since $y_N=1/N$, the definitions of $I$ and $F_\delta$ give the exact identity
\begin{align}
 &\bigl(\log N-I(\delta)\bigr)
 -\bigl(\log m-F_\delta(v_m)\bigr)\notag\\
 &\qquad=\sum_{n=m}^{N-1}\epsilon_n
 -\int_0^{1/N}\frac{\dd y}{B_{q,\Lam}(y)}.
 \label{eq:terminal-exact-telescope}
\end{align}
We estimate its terms in three disjoint regions.

\emph{Bottleneck and fixed compact region.}
Choose a fixed $\varepsilon>0$ below the bottleneck, uniformly for $q\in Q$,
and let $n_\varepsilon$ be the first index with $y_n\le\varepsilon$.
On the local bottleneck, the full-centre estimate
\eqref{eq:full-defect-cauchy} from Proposition \ref{prop:wedge} applies; on
the remaining fixed compact part, the field has a uniform positive lower
bound and a direct Taylor expansion gives the same bound.  Consequently
\begin{equation}\label{eq:terminal-bottleneck-defect}
 |\epsilon_n|
 \le\frac{C_Q}{n^2[\delta+(y_n-y_*)^2]}
 \le\frac{C_Q}{n^2\delta},
 \qquad m\le n<n_\varepsilon.
\end{equation}
Choose the scaled-step constant in (H5) no larger than that in
Lemma~\ref{lem:scaled-step}.  That lemma preserves the condition at every
step as the orbit approaches and exits the inner layer.  On the remaining
fixed compact region it follows by increasing the fixed index threshold,
since $\delta+(y_n-y_*)^2$ is bounded below there.
Consequently this entire region contributes at most $C_Q/(m\delta)$.

\emph{Intermediate region.}
Decrease $\varepsilon$ so that
\begin{equation}\label{eq:small-y-field}
 \frac12\le B_{q,\Lam}(y)\le\frac32,
 \qquad 0\le y\le\varepsilon,
\end{equation}
uniformly for $q\in Q$ and all sufficiently small $\delta$.  This is possible
because $B_{q,\Lam}(y)\to1$ as $y\downarrow0$, uniformly on compact
$q$-sets.  Let $n_A^{\mathrm{hit}}$ be the first index at or after $n_\varepsilon$ for which
$X_{n_A^{\mathrm{hit}}}\le A$.  In the region $n_\varepsilon\le n<n_A^{\mathrm{hit}}$, one has
$y_n<\varepsilon$ and $X_n>A$.  Uniform expansion of the exact map gives
\begin{align}
 y_n-y_{n+1}
 &=\frac1nB_{q,\Lam}(y_n)
 +O_Q\!\left(\frac1{n^2y_n}\right),
 \label{eq:intermediate-drift}\\
 |\epsilon_n|
 &\le\frac{C_Q}{n^2}\left(1+\frac1{y_n}\right).
 \label{eq:intermediate-defect}
\end{align}
Choose $A$ so large that the remainder in
\eqref{eq:intermediate-drift} is at most $1/(4n)$.  Then
$y_n-y_{n+1}\asymp_Q1/n$.  Since $y_{n_A^{\mathrm{hit}}}\ge1/n_A^{\mathrm{hit}}$, summing backwards yields
\begin{equation}\label{eq:intermediate-lower-y}
 y_n\ge c_Q\left(\log\frac{n_A^{\mathrm{hit}}}{n}+\frac1{n_A^{\mathrm{hit}}}\right),
 \qquad n_\varepsilon\le n\le n_A^{\mathrm{hit}}.
\end{equation}
For $n\le n_A^{\mathrm{hit}}/2$ the logarithm is bounded below and
$\sum n^{-2}y_n^{-1}\le C_Q\sum n^{-2}$.  In the final half write
$j=n_A^{\mathrm{hit}}-n$; then \eqref{eq:intermediate-lower-y} gives
$y_n\ge c_Q(j+1)/n_A^{\mathrm{hit}}$ and hence
\begin{align}
 \sum_{n=n_\varepsilon}^{n_A^{\mathrm{hit}}-1}\frac1{n^2y_n}
 &\le C_Q\left(\frac1{n_\varepsilon}
 +\frac1{n_A^{\mathrm{hit}}}\sum_{j< n_A^{\mathrm{hit}}/2}\frac1{j+1}\right)\notag\\
 &\le C_Q\left(\frac1m+\frac{\log n_A^{\mathrm{hit}}}{n_A^{\mathrm{hit}}}\right)
 \le C_Q\frac{\ell_m}{m}.
 \label{eq:intermediate-sum}
\end{align}
The last inequality uses $n_A^{\mathrm{hit}}\ge m$ and the eventual decrease of
$x\mapsto(1+\log x)/x$.  Equations
\eqref{eq:intermediate-defect}--\eqref{eq:intermediate-sum} bound the
intermediate contribution by $C_Q\ell_m/m$.

\emph{Terminal region.}
Lemma \ref{lem:terminal-layer} gives
$N-n_A^{\mathrm{hit}}=O_{Q,A}(1)$ and hence $n_A^{\mathrm{hit}}=N+O_{Q,A}(1)$.  Throughout these last steps
$y_n=O_{Q,A}(N^{-1})$.  The logarithmic increments, the integrals in
\eqref{eq:terminal-step-defect}, and the remaining tail in
\eqref{eq:terminal-exact-telescope} are therefore all $O_{Q,A}(N^{-1})$;
\eqref{eq:small-y-field} bounds the reciprocal denominator uniformly.
Combining the three regions in \eqref{eq:terminal-exact-telescope} gives
\[
 \left|\bigl(\log N-I(\delta)\bigr)
 -\bigl(\log m-F_\delta(v_m)\bigr)\right|
 \le C_Q\left(\frac1{m\delta}+\frac{\ell_m}{m}+N^{-1}\right),
\]
which is (H5), uniformly on $Q$, with $c_*=1$.
\end{proof}

\begin{proposition}[Uniform continuous bottleneck expansion]\label{prop:nolog}
Define
\begin{equation}\label{eq:I-power}
 I(q,\delta)=\int_0^{\Lstar(q)-\delta}
 \frac{\dd y}{B_{q,\Lstar(q)-\delta}(y)}.
\end{equation}
For every $Q\Subset(-1,\infty)$ there are
$\theta_Q\in(0,\pi/4)$ and $\varepsilon_Q>0$ such that, on the common sector
\begin{equation}\label{eq:xi-sector}
 \Sigma_Q=\left\{\xi\in\mathbb C:
 0<\lvert\xi\rvert<\varepsilon_Q,\quad
 \lvert\arg\xi\rvert<\theta_Q\right\},
\end{equation}
the following expansion holds at every fixed order $M$:
\begin{equation}\label{eq:I-expansion}
 I(q,\xi^2)=\frac{c(q)}{\xi}+J(q)
 +\sum_{k=1}^M\mu_k(q)\xi^k+O_{Q,M}(\xi^{M+1}),
\end{equation}
where $c,J,\mu_k$ are holomorphic near $Q$ and
\begin{equation}\label{eq:cq}
 c(q)=\frac{\pi}{\sqrt{a(q)b(q)}}
 =\pi\sqrt{2(1+q)\Lstar(q)}.
\end{equation}
There are no $\xi^k\log\xi$ terms.
\end{proposition}

\begin{proof}
Compactness permits a uniform $r>0$ such that
\[
 0<y_*(q)-2r<y_*(q)+2r<\Lstar(q)
\]
for $q\in Q$.  Split the integral into the local interval
$[y_*-r,y_*+r]$ and the two outer pieces.  On the right outer piece the
upper endpoint $\Lstar(q)-\delta$ is moving, but the denominator equals $1$
at that endpoint.  Hence \eqref{eq:moving-endpoint-rule} makes this piece
holomorphic in $(q,\delta)$.  The left piece is also holomorphic in these
parameters.  To see this uniformly at $y=0$, use
\[
 B_{q,\Lam}(y)=1+\frac{\Lam^{-q}y^{q+1}-y}{q}\qquad(q\ne0)
\]
and its $q=0$ continuation.  Choose
$-1<q_-<\min\{\inf Q,0\}$ and shrink the complex neighbourhood of $Q$ so that
$\operatorname{Re}q\ge q_-$ there.  Using
the integral form of the removable quotient at $q=0$, every fixed parameter
derivative is then dominated near zero by
$C_{Q,m}y^{q_-+1}(1+|\log y|^m)$, which is integrable.  This supplies one
common majorant for differentiation under the integral sign.

Put $w=y-y_*(q)$.  Weierstrass preparation in $(q,\xi,w)$ gives
\begin{equation}\label{eq:weierstrass-roots}
 B_{q,\Lstar(q)-\xi^2}(y_*(q)+w)
 =U(q,\xi,w)(w-w_+(q,\xi))(w-w_-(q,\xi)),
\end{equation}
where $U$ is nonzero and
\begin{equation}\label{eq:root-separation}
 w_\pm(q,\xi)=\pm i\sqrt{\frac{a(q)}{b(q)}}\,\xi+O_Q(\xi^2).
\end{equation}
Shrink a complex neighbourhood of $Q$, $\varepsilon_Q$, and $\theta_Q$ if
necessary.  Then for every $\xi\in\Sigma_Q$, the root $w_+$ lies inside the
fixed upper half-disc $|w|<r$, $\operatorname{Im}w>0$, the root $w_-$ lies
outside it, and neither root meets its boundary.

Let $\Gamma_+$ be the positively oriented contour formed by the segment
$[-r,r]$ followed by the upper semicircle
\begin{equation}\label{eq:upper-return-arc}
 \gamma_+(\vartheta)=re^{i\vartheta},
 \qquad 0\le\vartheta\le\pi,
\end{equation}
traversed from $r$ back to $-r$.  The residue theorem gives the exact identity
\begin{align}
 &\int_{-r}^{r}
 \frac{\dd w}{B_{q,\Lstar(q)-\xi^2}(y_*(q)+w)}\notag\\
 &\quad=2\pi i\,
 \operatorname*{Res}_{w=w_+(q,\xi)}
 \frac1{B_{q,\Lstar(q)-\xi^2}(y_*(q)+w)}
 -\int_{\gamma_+}
 \frac{\dd w}{B_{q,\Lstar(q)-\xi^2}(y_*(q)+w)}.
 \label{eq:contour-decomposition}
\end{align}
The return-arc integral stays a fixed distance from both coalescing roots and
is therefore holomorphic in $(q,\delta)$.  Together with the outer pieces,
\eqref{eq:contour-decomposition} yields
\begin{equation}\label{eq:residue-plus-even}
 I(q,\xi^2)=2\pi i\,\mathcal R_+(q,\xi)+H(q,\xi^2),
\end{equation}
where $H$ is holomorphic and $\mathcal R_+$ is the residue at $w_+$.  By
\eqref{eq:root-separation}, $\mathcal R_+$ has a simple pole at $\xi=0$ and a
convergent Laurent expansion in integer powers of $\xi$; its leading term is
$(2i\sqrt{a(q)b(q)}\,\xi)^{-1}$.  Thus
$2\pi i\mathcal R_+=\dfrac{\pi}{\sqrt{a(q)b(q)}\,\xi}+O_Q(1)$.  Formula
\eqref{eq:residue-plus-even} proves \eqref{eq:I-expansion}, including the
absence of all $\xi^k\log\xi$ terms.  It also shows explicitly that moving
endpoints are absorbed into the holomorphic even term $H(q,\xi^2)$.
\end{proof}

\medskip
\begin{proof}[Proof of Theorem \ref{thm:main}]
First carry out the construction on an arbitrary compact interval
$Q\Subset(-1,\infty)$ with nonempty interior.
Lemma \ref{lem:capture} and Propositions \ref{prop:wedge}--\ref{prop:nolog}
verify (H1)--(H6) uniformly on $Q$, with the exceptional endpoint treated as
specified above.  Theorem \ref{thm:abstract} gives
\eqref{eq:main-expansion}, while Propositions \ref{prop:wedge} and
\ref{prop:nolog} make the discrete and continuous phase data holomorphic
near $Q$.

Restrict each local construction to the interior of its interval.  At every
point in an overlap, two normalized expansions describe the same sequence
$\Lstar(q)-\lambda_N(q)$.  If their phases differ by $s$, comparison at
cubic order gives $0=2sA_2(q)$; since $A_2(q)>0$, $s=0$.
Successive comparison of powers then identifies every $A_j(q)$.
These interiors cover $(-1,\infty)$, so the functions patch uniquely to
real-analytic $\kappa,A_j$ on the whole interval.  Every compact subset is
contained in one compact interval, which gives the asserted uniformity for
arbitrary compact $Q$, including sets with empty interior.
Formula \eqref{eq:cq} gives
$A_2=c^2=\mathcal K(q)$.  The coefficient $A_3$ vanishes by the phase-normalized
asymptotic inversion.  The computation of $A_4$ is completed in the next section.
\end{proof}

\section{Local coefficients and the Carleman calculation}
\label{sec:coefficients}

\begin{lemma}[The first local odd coefficient]\label{lem:mu1-general}
For $q\ne0$, the local residue expansion in Proposition \ref{prop:nolog}
satisfies
\begin{equation}\label{eq:mu1-general}
 \mu_1(q)c(q)
 =-\frac{\pi^2}{6}(2q^2+5q+5).
\end{equation}
The identity extends analytically to $q=0$.
\end{lemma}

\begin{proof}
A nonzero zero $z$ of $B_{q,\Lam}$ satisfies
\[
 \Lam=\Phi_q(z),\qquad
 \Phi_q(z)=\left(\frac{z^{q+1}}{z-q}\right)^{1/q},
\]
with the local analytic branch fixed near $z=y_*=1+q$.  At such a zero,
\[
 \partial_yB_{q,\Lam}(z)
 =\log_q(z/\Lam)+(z/\Lam)^q
 =\frac{z-y_*}{z}.
\]
Thus, on writing $w=z-y_*$, the residue of $B^{-1}$ is exactly
\begin{equation}\label{eq:general-residue}
 \operatorname*{Res}_{y=z}\frac1{B_{q,\Lam}(y)}
 =\frac{y_*+w}{w}.
\end{equation}

Expansion of the root-parameter function at its critical point gives
\begin{equation}\label{eq:Phi-expansion}
 \frac{\Phi_q(y_*+w)}{\Lstar}
 =1+\alpha_2w^2+\alpha_3w^3+\alpha_4w^4+O(w^5),
\end{equation}
where
\begin{equation}\label{eq:alphas}
 \alpha_2=\frac1{2(1+q)},\qquad
 \alpha_3=-\frac{q+2}{3(1+q)^2},\qquad
 \alpha_4=\frac{2q^2+7q+7}{8(1+q)^3}.
\end{equation}
In particular,
\begin{equation}\label{eq:alpha-combination}
 4\alpha_2\alpha_4-3\alpha_3^2
 =\frac{2q^2+5q+5}{12(1+q)^4}.
\end{equation}
Put
\[
 \beta_2=\Lstar\alpha_2,\qquad
 \beta_3=\Lstar\alpha_3,\qquad
 \beta_4=\Lstar\alpha_4.
\]
The upper-half-plane solution of
$\Phi_q(y_*+w_+)=\Lstar-\xi^2$ is obtained by coefficient comparison:
\begin{equation}\label{eq:general-root}
 w_+(\xi)=\frac{i}{\sqrt{\beta_2}}\xi
 +\frac{\beta_3}{2\beta_2^2}\xi^2
 +\frac{i(\beta_2\beta_4-5\beta_3^2/4)}
 {2\beta_2^{7/2}}\xi^3+O(\xi^4).
\end{equation}
By the exact contour decomposition \eqref{eq:residue-plus-even},
\begin{equation}\label{eq:odd-residue-identity}
 I(q,\xi^2)
 =2\pi i\left(1+\frac{y_*}{w_+(\xi)}\right)+H(q,\xi^2),
\end{equation}
where $H$ is holomorphic in $\xi^2$.  Thus the constant $1$ in the residue is
regular, and every odd Laurent coefficient of the real passage time comes
from
\[
 \operatorname{Re}\left(\frac{2\pi i\,y_*}{w_+(\xi)}\right).
\]
No branch of an endpoint logarithm is being suppressed here: the identity is
the residue theorem on the explicitly oriented contour
\eqref{eq:upper-return-arc}.
Inverting \eqref{eq:general-root} gives
\begin{align}
 c(q)&=2\pi y_*\sqrt{\beta_2}
 =\pi\sqrt{2(1+q)\Lstar},\label{eq:c-residue}\\
 \mu_1(q)&=
 -\frac{\pi y_*}{4\beta_2^{5/2}}
 (4\beta_2\beta_4-3\beta_3^2).
 \label{eq:mu-residue}
\end{align}
Multiplying these formulas and using
\eqref{eq:alphas}--\eqref{eq:alpha-combination} proves
\eqref{eq:mu1-general}.  Both sides are real analytic in $q$ through the
removable point $q=0$, so the identity extends there.
\end{proof}

Combining \eqref{eq:mu1-general} with $A_4=2\mu_1c^3$ and
$c^2=\mathcal K(q)$ gives
\eqref{eq:A4q}; \eqref{eq:A4t} follows from $q=t/(1-t)$.  No global orbit
data enter this coefficient.  For completeness, and to avoid using analytic
continuation as the only check at the removable parameter, we now calculate
$q=0$ independently.

\begin{lemma}[Direct Carleman residue]\label{lem:carleman-residue}
At $q=0$,
\begin{equation}\label{eq:carleman-mu}
 c(0)=\pi\sqrt{2e},\qquad
 \mu_1(0)=-\frac{5\pi}{6\sqrt{2e}},
\end{equation}
and hence $A_4(0)=-(10/3)e\pi^4$.
\end{lemma}

\begin{proof}
Set $\delta=\xi^2$.  The Carleman denominator is
\[
 B_{0,e-\xi^2}(y)=1+y\log\frac{y}{e-\xi^2}.
\]
Its upper-half-plane local zero has the expansion
\begin{equation}\label{eq:carleman-zero}
 y_+(\xi)=1+i\sqrt{\frac2e}\,\xi
 -\frac4{3e}\xi^2
 -\frac{17i\sqrt2}{36e^{3/2}}\xi^3+O(\xi^4),
\end{equation}
obtained by direct substitution.  At a zero,
$\log(y_+/(e-\xi^2))=-1/y_+$, so
\begin{align}
 \operatorname*{Res}_{y=y_+(\xi)}\frac1{B_{0,e-\xi^2}(y)}
 &=\frac{y_+}{y_+-1}\notag\\
 &=-\frac{i\sqrt{2e}}{2\xi}+\frac13
 +\frac{5i\sqrt2}{24\sqrt e}\xi+O(\xi^2).
 \label{eq:carleman-residue}
\end{align}
The contour identity \eqref{eq:residue-plus-even} says that the passage time is
$2\pi i$ times this residue plus a function holomorphic in $\xi^2$; the latter
changes neither the $\xi^{-1}$ nor the $\xi$ coefficient.  Formula \eqref{eq:carleman-mu}
follows, and $2\mu_1c^3=-(10/3)e\pi^4$.
\end{proof}

\begin{remark}[Where discreteness first enters]
The phase itself contains the global discrete defect
$\tau_0$.  After centring by $\kappa=-(J+\tau_0)$, however, the coefficients
$A_2$ and $A_4$ are determined by the local continuous bottleneck and
$A_3=0$.  The first post-phase coefficient that depends on a nonconstant
discrete response is
\[
 A_5=2(\mu_2+\tau_1)c^4.
\]
This local/global separation is one reason to retain the phase-normalized
form of the expansion.
\end{remark}

\section{Classical projections and a limited numerical check}
\label{sec:projections}

If $s=1/t$ and $x_k=a_k^s$, then
\begin{equation}\label{eq:projection-identity}
 P_t(x_1,\ldots,x_n)
 =\left(\frac1n\sum_{k=1}^na_k\right)^s.
\end{equation}
Thus $s>1$ produces the classical positive Hardy quotient, $s<0$ the
negative-power quotient, and $t=0$ the geometric-mean limit.  To make the
positive-branch normalization explicit, put $s=p>1$ and
$p'=p/(p-1)$.  Then \eqref{eq:projection-identity} shows that
$\Lam_N(1/p)$ is exactly the best constant $\mu_N$ in the classical finite
$\ell^p$ Hardy inequality.  For $p\ge2$, Gao's Theorem~1.1
\cite{GaoHardy}, specialized to the unweighted sequence $\lambda_k=1$ and
hence to $L=1$, gives
\begin{equation}\label{eq:gao-unweighted}
 \mu_N
 =(p')^p-\frac{2\pi^2(p')^{p+1}}{(\log N)^2}
 +O((\log N)^{-3}).
\end{equation}
Since $t=1/p$ corresponds to $q=1/(p-1)$ and
\[
 A_2\!\left(\frac1{p-1}\right)=2\pi^2(p')^{p+1},
\]
the positive-branch projection of \eqref{eq:A2t} agrees with
\eqref{eq:gao-unweighted}.  This comparison uses Gao's theorem itself, not
the distinct formula attributed to Ackermans in the introduction of
\cite{GaoHardy}.  At $t=0$, \eqref{eq:A2t} gives de Bruijn's correction.

There is a stronger benchmark at $p=2$, where
\[
 p=2\quad\Longleftrightarrow\quad t=\frac12
 \quad\Longleftrightarrow\quad q=1,
 \qquad \Lam_N(1/2)=\lambda_N(1).
\]
This common constant is the norm of the finite Hilbert $L$-matrix used in
\cite{Stampach2022}, so the comparison concerns exactly the same finite
problem.  Put
\[
 \kappa_1=\gamma+6\log2,
 \qquad X=\log N+\kappa_1.
\]
Here $\gamma$ is Euler's constant.
Equation~(42) and Appendix~A.2 of \cite{Stampach2022}, rewritten in this
centred scale, give
\begin{equation}\label{eq:p2-benchmark}
 4-\lambda_N(1)
 =\frac{16\pi^2}{X^2}-\frac{64\pi^4}{X^4}
 +\frac{416\pi^4\zeta(3)}{3X^5}+O(X^{-6}).
\end{equation}
The full coefficient formula in Theorem~1.4 and (1.14)--(1.16) of
\cite{DimitrovGadjevIsmail2024}, by contrast, amounts to expanding
$16\pi^2/(X^2+4\pi^2)$ for the defect.  It agrees with
\eqref{eq:p2-benchmark} through $X^{-4}$ but has no $X^{-5}$ term.  Thus the
two published higher-order assertions conflict at the fifth centred order.

For completeness, the nonzero term in \eqref{eq:p2-benchmark} can be checked
directly from the Gamma phase in the characteristic equation used in
\cite{Stampach2022}.  On the branch continuous at the origin, set
\[
 \vartheta(x)=\operatorname{Im}\!\left(
 3\log\Gamma(1/2-ix)-\log\Gamma(1-2ix)\right).
\]
The standard polygamma values give
\[
 \vartheta(x)=\kappa_1x-\frac{13}{3}\zeta(3)x^3+O(x^5).
\]
Consequently the smallest spectral parameter satisfies
\[
 Xx-\frac{13}{3}\zeta(3)x^3+O(x^5)=\pi,
 \qquad
 x=\frac\pi X+\frac{13\pi^3\zeta(3)}{3X^4}+O(X^{-6}).
\]
Substitution into
$4-\lambda_N(1)=16x^2/(1+4x^2)$ recovers
\eqref{eq:p2-benchmark}.  In particular, uniqueness of the normalized phase
and \eqref{eq:abstract-coefficients} imply
\begin{equation}\label{eq:p2-phase-data}
 \kappa(1)=\gamma+6\log2,\qquad
 A_2(1)=16\pi^2,\qquad A_4(1)=-64\pi^4,\qquad
 A_5(1)=\frac{416}{3}\pi^4\zeta(3),
\end{equation}
and, since $c(1)=4\pi$,
\begin{equation}\label{eq:p2-global-response}
 \mu_2(1)+\tau_1(1)=\frac{13}{48}\zeta(3).
\end{equation}
The continuous part can also be separated explicitly.  For $q=1$,
\[
 I(1,\delta)
 =4\sqrt{\frac{4-\delta}{\delta}}
 \arctan\sqrt{\frac{4-\delta}{\delta}},
\]
and hence, with $\xi=\sqrt\delta$,
\[
 I(1,\xi^2)=\frac{4\pi}{\xi}-4-\frac\pi2\xi
 +\frac13\xi^2+O(\xi^3).
\]
Therefore
\[
 \mu_2(1)=\frac13,
 \qquad
 \tau_1(1)=\frac{13}{48}\zeta(3)-\frac13.
\]
These identities are external $p=2$ consistency data, not inputs to the
proof of Theorem~\ref{thm:main}.  The agreement through fourth order and the
resolved fifth-order value are benchmarks for the unified mechanism, not
novelty claims.

At $q=0$, the exact map admits a stable backward form:
\begin{equation}\label{eq:carleman-backward}
 y_n=\frac1n+y_{n+1}
 \exp\left(\frac{\log y_{n+1}-\log\Lam}{n}\right),
 \qquad y_N=\frac1N.
\end{equation}
Numerically, we solve the backward residual
$R_N(\Lam)=y_1(\Lam)-\Lam$ by a bracket-protected Newton iteration while
propagating $\partial_\Lam y_n$.  An x87 extended-precision implementation
(64-bit significand) supplies the archived data through $N=10^7$.  An
independent IEEE binary128 implementation at
$N=10^3,10^4,10^5,10^6$ agrees with it within $8.2\times10^{-19}$; its
backward residuals are of order $10^{-31}$.  At $N=10^7$ the archived value is
\[
 \lambda_N(0)=\Lam_N(0)=2.5885777308970364\ldots.
\]
On $N\ge10^5$, the spread of the leading-only inferred phase is
$0.0964627$; inserting the theoretical $A_4(0)$ reduces it to $0.0464549$, a
ratio of $0.4816$.  In seventh-order fits with $A_4$ left free, several data
windows recover the theoretical coefficient within about $0.2$--$0.5$
percent.  With $A_4$ fixed, fits through orders $7$--$9$ place
$\kappa(0)$ between $3.84844$ and $3.84868$.  The design-matrix condition
number, however, grows from about $10^4$ to $10^{10}$ and the fitted $A_5$
remains truncation-sensitive.  We therefore report no numerical value for
$A_5$.  In particular, the agreement of the free-$A_4$ fits is diagnostic,
not an independent verification of Lemma~\ref{lem:mu1-general}:
on this range the $X^{-4}$ and $X^{-5}$ columns are nearly collinear.  The
code, data, fitting audit, and symbolic residue check used for these
diagnostics are supplied in the accompanying supplementary archive; none of
them is used in the proof.

At the Carleman point the constructive phase formula is
\[
 \kappa(0)=-J(0)-\tau_0(0),\qquad
 \tau_0(0)=\sum_{n\ge1}d_{n,0,0}(v_n^*(0)),
\]
and Theorem \ref{thm:abstract} proves that the defect series is absolutely
convergent.  It can therefore be evaluated in principle by iterating the
critical orbit, estimating the summable tail, and computing $J(0)$.  The
present numerical audit does not perform that direct defect-series
evaluation; instead, finite-section fits give the diagnostic value
\[
 \kappa(0)\approx3.8486.
\]
Thus this decimal is neither a certified interval nor a certified evaluation
of $-(J(0)+\tau_0(0))$, and it is not used in any proof.  We do not obtain a
reduction of $\tau_0(0)$ to Euler--Maclaurin constants, Gamma or Barnes
functions; deciding whether such a reduction exists is a separate arithmetic
problem about the exact critical discrete orbit.

\section{Scope and open directions}

Theorem \ref{thm:main} is uniform on compact subsets of the full analytic
parameter interval.  The loss of uniformity at its two ends has a concrete
scale mechanism.  If $\varepsilon=1+q\downarrow0$, then
\[
 y_*=\varepsilon,\qquad \Lstar\to1,
 \qquad b(q)=\frac1{2\varepsilon},
 \qquad c(q)\sim\pi\sqrt{2\varepsilon}.
\]
The fixed local neighbourhood shrinks with $y_*$, and the terminal scale
$y\asymp1/n$ reaches the bottleneck at $n\asymp\varepsilon^{-1}$.  Thus when
$N\varepsilon$ is not large, the terminal layer and the bottleneck can no
longer be separated by the constants used in (H3)--(H5).

At the other endpoint,
\[
 y_*\sim q,\qquad \Lstar=q+\log q+O(1),
 \qquad c(q)\sim\pi\sqrt2\,q.
\]
The fixed-$q$ inversion has $\sqrt\delta\asymp c(q)/\log N$, so its local
small-parameter regime is not uniform once $q$ is comparable with $\log N$.
At the limiting arithmetic mean itself one has exactly
\[
 \Lam_N(1)=\sum_{n=1}^N\frac1n\sim\log N,
\]
which also shows that the limits $N\to\infty$ and $q\to\infty$ cannot simply
be interchanged.  Endpoint double-scaling problems therefore require separate
normalizations and are not consequences of Theorem \ref{thm:main}.

The all-order proof controls each fixed jet separately; the constants $C_j$
in Lemma \ref{lem:discrete-jet} are not estimated as $j\to\infty$.  Hence the
present argument proves neither convergence nor divergence of the formal
series, and in particular it does not establish a Gevrey--1 bound.  Factorial
growth and Borel summability are natural questions, not conclusions of this
paper.  Other natural problems are the variable-step phase theory generated
by matched weights and the structure of the global coefficients from $A_5$
onward.  In the coordinates used here, general matched weights produce a
related exact recurrence with two distinct steps.  Indeed, let $w_n>0$,
$W_n=\sum_{k\le n}w_k$, and
\[
 M_n=\left(W_n^{-1}\sum_{k\le n}w_kx_k^t\right)^{1/t},
\]
with the weighted geometric mean at $t=0$.  For the quotient
$\sum_{n\le N}w_nM_n/\sum_{n\le N}w_nx_n$, the Euler equations at a
positive stationary point give
\[
 \Lam x_k^{1-t}=\sum_{n=k}^N\frac{w_n}{W_n}M_n^{1-t}.
\]
Set $y_n=\Lam(x_n/M_n)^{1-t}$,
$a_n=w_n/W_n$, and $b_n=w_{n+1}/W_n$.
Subtracting consecutive Euler equations and using the weighted prefix
update yields, for $q\ne0$,
\begin{equation}\label{eq:weighted-two-step}
 y_{n+1}=(y_n-a_n)
 \left[\frac{1+b_n}
 {1+b_n((y_n-a_n)/\Lam)^q}\right]^{1/q}.
\end{equation}
At $q=0$ its removable value is
$(y_n-a_n)^{1/(1+b_n)}\Lam^{b_n/(1+b_n)}$.
Constant weights give $a_n=b_n=1/n$ and recover \eqref{eq:exact-map};
general weights need not satisfy $a_n=b_n$.
Any reduction to a single-step map by a new coordinate or clock, together
with the necessary mesh and support conditions, requires a separate
argument and is not asserted by the present theorem.
For more general homogeneous means, the same quadratic bottleneck
may survive without an exact $q$-map; that would belong to a broader
universality class rather than to the exact power-mean family proved here.

\section*{Declaration of generative AI and AI-assisted technologies}

During the research and preparation of this manuscript, between 1 August and
28 September 2026, the author used OpenAI ChatGPT and Codex, accessed primarily
through their official web interfaces, with occasional use of the ChatGPT
desktop application.  The principal OpenAI models used for mathematical
research were GPT-5.6 Sol at the maximum reasoning setting and GPT-6 Astra.
These tools were used for exploratory derivations, proof auditing, symbolic
and numerical cross-checks, and editorial revision.  The author also used xAI
Grok 4.6 and Google Gemini 3.8 Flash, primarily through their official web
interfaces, for independent cross-auditing and deep literature and citation
searches.  All AI-assisted material incorporated into the manuscript was
critically reviewed and edited by the author.  The author made the final
decisions about the mathematical arguments, computations, cited sources,
conclusions, and wording, and assumes full responsibility for the content.  No
AI system is listed as an author.

\end{document}